\documentclass[11pt,onecolumn]{article}
\usepackage{amscd}
\usepackage{times}
\usepackage{amsmath}
\usepackage{amssymb}
\usepackage{xspace}
\usepackage{theorem}
\usepackage{graphicx}
\usepackage{ifpdf}
\usepackage{url,hyperref}
\usepackage{latexsym}
\usepackage{euscript}
\usepackage{xspace}
\usepackage[all]{xy}
\usepackage{color}
\usepackage{makeidx}
\usepackage{tikz}
\usepackage{enumitem}
\long\def\remove#1{}
\newtheorem{theorem}{Theorem}[section] % section

\newtheorem{obs}[theorem]{Observation}
\newtheorem{corollary}[theorem]{Corollary}

\newtheorem{definition}[theorem]{Definition}
\newtheorem{proposition}[theorem]{Proposition}
\newenvironment{proof}{{\em Proof:}}{\hfill{\hfill\rule{2mm}{2mm}}}

\newcommand {\mm}[1] {\ifmmode{#1}\else{\mbox{\(#1\)}}\fi}

\newcommand{\img}{\mathrm img}
\newcommand{\supp}{\mathrm supp}
\newcommand{\coker} {\mathrm coker}

\newcommand{\rank}                {\mm {\rm rank}}

\newcommand{\cancel}[1]

\begin{document}

\title{Alternative to Morse-Novikov Theory  for a closed 1-form (I)}

\author{
Dan Burghelea  \thanks{
Department of Mathematics,
The Ohio State University, Columbus, OH 43210,USA.
Email: {\tt burghele@math.ohio-state.edu}}
}
\date{}

\date{}

\maketitle

\hskip 1.5in Dedicated to  the memory of Stefan Papadima
\vskip .2in
\begin{abstract}
This paper extends the Alternative to Morse-Novikov Theory we have proposed in \cite{B} from real-and angle-valued map to closed 1-forms.
For a topological closed 1-form $\omega$ on a compact ANR $X,$  a concept generalizing  closed differential 1-form on a compact manifold,
under  the mild hypothesis of tameness, a field $\kappa$ and a non-negative integer $r$ we propose  two configurations $\boldsymbol \delta^\omega_r: \mathbb R\to \mathbb Z_{\geq 0}$ and $\boldsymbol\gamma^\omega_r: \mathbb R_{>0}\to \mathbb Z_{\geq 0}$ which recover Novikov-Betti numbers 
and  the Novikov complex associated with a Morse closed 1-form with non degenerated zeros.
Precisely, the sum of the multiplicities of the points in the support of  $\boldsymbol\delta^\omega_r$ equals the $r-$th Novikov-Betti number  and that of the points in the support of  $\boldsymbol\gamma^\omega_r$  equals the rank of the boundary map in the Novikov complex. 
We formulate the basic properties of these configurations, the stability and the Poincare duality when $X$ is a $\kappa-$orientable closed topological manifold which in full generality will be proven in the second part of this work.
\end{abstract}

\maketitle 

\setcounter{tocdepth}{1}
\tableofcontents

\section{Introduction}
\label{intro}
In this paper \footnote { The present version of this this paper is prompted by the discovery of a mistake in the proof of Theorem 3.2 page 723 of the printed version  of this work in  EJM  2020 no 6.   
Fortunately all results but Theorem 3.2 as formulated remain true by essentially the same arguments with minor improvements once an additional requirement (  hypothesis  H, %" $\mathbb J^\xi_r$ is injective" 
cf section 2 for definition), be added to the tameness  of the TC1-form.  This requirement  is conjecturally superfluous  and indeed superfluous  for all TC1-forms whose cohomology class  $\xi$ belongs to $img (H^1(M;\mathbb Z)\to H^1 (M;\mathbb R)$  and for many other cases of pairs $(M; \xi \in H^1(M;\mathbb R))$ as discussed in subsection 2.2 } 
we define the configurations $\delta^\omega$ and $\gamma^\omega$  for a tame topological closed one form $\omega$ on a compact ANR. They are analogues of the configurations $\delta^f_r$  and $\gamma^f_r,$  previously defined in \cite {BH} for a tame real- or angle-valued maps $f.$ %$ to  a tame topological  closed 1-forms $\omega$ (cf. section \ref{S2} for definition). 
As a consequence we extend our Alternative to Morse-Novikov Theory (cf. \cite{B}) from tame angle-valued maps to  tame topological closed 1-form on a compact ANR.The concept {\it topological closed 1-form}, abbreviated TC1-form,  is  a generalization of closed differential 1-form on a smooth manifold. As in the case of real- or angle-valued maps (which corresponds to the case of closed 1-form of degree of irrationality zero or  one) we will analyze the real-valued map $f^\omega:\tilde X\to \mathbb R,$ a lift of $\omega,$ defined on the total space of the associated principal $\Gamma= \mathbb Z^k-$covering $\pi:\tilde X\to X$ associated to the cohomology class $[\omega]\in H^1(X;\mathbb R)$       
with $k$ the {\it degree of irrationality} of $\omega.$ 

It might appear that this should be a routine  extension  of the case $k=1$ but this is not quite so because: 

1.  the map $f^\omega$ is never proper when degree of irrationality $k$  is greater than $1$, 
so the  homology vector spaces which are involved in, finite dimensional in the case $k\leq 1,$   might be infinite dimensional  in the case $k\geq 2,$

2.  the set of critical values of $f^\omega$ is not discrete when $k>1$
but the opposite, always dense if not empty;  
 the approach of Zig-Zag persistence based on graph representations,  cf. \cite {BD11}, is apparently not applicable. 
 
 However, the tameness of the lift $f^\omega$ of $\omega,$ ,see secytion \ref {S2} for definition) and the fact that the group $\Gamma-$ defined by the form $\omega$ appears as a subgroup of $\mathbb R$ 
make the approach described in 
sections 6 and 7 of \cite {B} \footnote {initiated in \cite {BH}}  applicable.  
Ultimately this  leads to the finite configuration $\boldsymbol \delta^\omega_r$ and $\boldsymbol \gamma^\omega_r$ of points in $\mathbb R$ and $\mathbb R_+$   derived however from the $\mathbb Z_{\geq 0}-$valued maps  $\boldsymbol \delta^{f^\omega}_r$ and $\boldsymbol \gamma^{f^\omega}_r$ which are not configurations. 

To prove our result we consider in section \ref{S3} an apparently new definition of Novikov-Betti numbers based on the lifts of $\omega$ and verify in section \ref{S6} that 
this definition is equivalent to the standard ones (cf. \cite {F} for definitions). 
We are unaware if it  already exists in literature.  

 For the configurations $\boldsymbol \delta_r^\omega$ and $\boldsymbol \gamma^\omega_r$ one can  prove a stability property and a Poincar\'e duality property similar to Theorems 1.3 and 1.5 in \cite {BH} 
and, in view of the stability property of the assignment 
$\omega \rightsquigarrow  \boldsymbol \gamma^\omega_r$  (cf. Theorem \ref{TS}), show that the configurations $\boldsymbol \delta^{\omega}_r$ can be actually defined for any  TC1-form, not necessary tame.  It also can 
be refined in the spirit  of \cite {B1}and \cite {B2} to an assignment with values $\kappa[\Gamma]-$modules but we are not interested in this aspect nor in its  implications at this time.

The main results about the configurations $\boldsymbol \delta^\omega_r$ and $\boldsymbol \gamma^\omega_r$ are stated in Theorems \ref{TT} %and  proven in this paper, while 
Theorems \ref{TP} and \ref{TS}. In this paper (part I) only Theorem \ref{TT} is proven entirely, the other two  will be established in part II and part III of this work. 

To formulate them we note : 

For a fixed field $\kappa,$ $X$ a compact ANR,  $\xi\in H^1(X;\mathbb R),$ denote by  $\beta^N_r(X;\xi)$  the $r-$th Novikov-Betti number, $r=0,1,2,\cdots$ %(cf \cite {F}  
%and subsection \ref{SS33}), 
and by $\mathcal Z^1(X;\xi)$ the set of topological closed 1-form on $X$ in the cohomology class $\xi$ (cf. definition section \ref {S2} ) equipped with the compact open topology.  

For a closed TC-1 form $\omega$ denote by $CR(\omega):= \{ t= c'-c''\mid c',c"\in CR(f^\omega)\}$ with $CR(f^\omega)$ the set of critical values of $f^\omega$ a lift of $\omega.$ Clearly this set of real numbers is independent of the lift $f^\omega$. % and it is $\Gamma-$invariant.

As described in Section 2 the cohomology class $\xi \in H^1(X;\mathbb R)$ defines  two $\kappa[\Gamma]-$modules, $H^{inv}(X;\xi)$ and $H^{dir}(X; \xi)$ and a $\kappa[\Gamma]-$linear map $ J^\xi _r : H^{inv}(X;\xi) \to H^{dir}(X; \xi)$ 
 conjecturally always  injective and verified to be injective  for any $\xi$  in the image of $H^1(X;\mathbb Q) \to H^1(X;\mathbb R),$ as shown in Section 2.

For a space $Y$ and closed subset $K\subset Y$ one denotes by $Conf_N(Y)$ the space of configurations of total cardinality $N$ equipped with the collision topology and by $Conf(Y\setminus K)$ the space of configuration of points in $Y\setminus K$ equipped with the {\it bottleneck topology}, all these topologies  described in sections \ref {S2} and \ref {S3} below.  In this paper $Y= \mathbb R$ and $K= (-\infty, 0].$
\newpage 

\begin {theorem} \label {TT}\ 

\begin{enumerate}
\item If  $\omega$ is a weakly tame topological  closed 1-form on $X$ (cf. section \ref{S2} for definition)  then %one has   

$$\sum_{t\in \mathbb R} \boldsymbol \delta^\omega_r (t) = \beta^N_r(X;\xi(\omega))$$ where $\xi(\omega)=[\omega]$ denotes  the cohomology class determined by $\omega.$ 

The support of $\boldsymbol \delta_r^\omega$ and $\boldsymbol \gamma_r^\omega$ are real numbers  in $CR(\omega).$

\item If $X= M^n$ is a closed smooth manifold of dimension $n,$ $\omega$ is a  tame  closed differential  1-form  (cf definition in section 2) in the cohomology class $\xi$  with all zeros of Morse  type and $c_r(\omega)$ denotes the number of zeros of Morse index $r$ then 

$$c_r(\omega)= \sum_{t\in \mathbb R} \boldsymbol \delta^\omega_r (t) + 
\sum_{t\in \mathbb R_+} \boldsymbol \gamma^\omega_r (t) + 
%\boxed{ ^-\lambda^\omega_r}  +\sum_{t\in \mathbb R_+} 
\boldsymbol \gamma^\omega_{r-1} (t) 
+ \boxed { \lambda^\omega_{r-1}}$$  with  $\lambda_{r-1}^\omega=0$ if $J_{r-1}^{\xi}$ is injective. The definition of $\lambda^\omega_r$ is given in section \ref {S4} (\ref {EE11} ).    
\end{enumerate}
\end{theorem}

Item 1. explains in what sense   $\boldsymbol \delta^\omega_r$ refines  the Novikov-Betti numbers.

Since a chain complex of finite dimensional vector spaces is, up to a non canonical isomorphism, completely determined by the dimensions of its homology vector spaces  and the dimensions of its components Item 2. explains to what extent  the  configurations $\boldsymbol \delta^\omega_r$ and  $\boldsymbol \gamma^\omega_r$ supplemented by the numbers $\lambda^\omega_r$ if nonzero) provide together a refinement of the Novikov complex when considered over the Novikov field.  Note that the formula in Item 2. gives the rank of $d_r: C_r\to C_{r-1},$ the boundary map in the Novikov complex.
% , at least in the case of $\omega$ of degree of irrationality one.
Precisely $$\rank \  d_r= \sum _{t\in \mathbb R_{>0}}\boldsymbol \gamma^\omega_r(t).$$     

The extension of this result  from smooth Morse closed one forms on a closed manifold to tame TC1-form on a compact ANR should be regarded as a considerable  weakening  of the hypothesis '' all zeros of $\omega$ are of Morse type" 
in order to recover all informations provided by the Novikov complex. 

\begin{theorem}\label {TP}\

Suppose $M$ is a closed topological manifold and $\omega\in  \mathcal Z^1_t( X;\xi),$
\begin {enumerate}
\item $\boldsymbol  \delta^\omega_r(t)= \boldsymbol  \delta^{\omega}_{n-r} (-t),$
\item $\boldsymbol \gamma^\omega_r(t)= \boldsymbol \gamma^{-\omega}_{n-r-1} (t).$
\end{enumerate}
\end{theorem}

Let $\mathcal Z^1_t( X;\xi) \subset \mathcal Z(X;\xi)$ denote the space of tame topological closed 1-forms with the topology induced from the compact open topology  on $\mathcal Z^1(X;\xi),$ \footnote { $\mathcal Z^1_t( X;\xi)$ is dense in $\mathcal Z(X;\xi)$} 
defined in Section \ref{S2}.  The topology on $Conf_{\beta^N_r(X:\xi)}(\mathbb R)$ is the collision topology and  the topology on $Conf (\mathbb R_+),$ with $\mathbb R_+$ viewed as $Y\setminus K$ for $Y= \mathbb R$ and $ K=(-\infty, 0],$ is the bottleneck topology described in subsection (\ref {SS31}).

\begin{theorem} \label {TS}\
\begin{enumerate}
\item The assignment $\boldsymbol  \delta_r : \mathcal Z^1_t( X;\xi)\rightsquigarrow Conf_{\beta^N_r(X:\xi)}(\mathbb R)$ is continuous  and extends to a continuous assignment on the entire $\mathcal Z^1(X;\xi).$
\item The assignment $\boldsymbol  \gamma_r : \mathcal Z^1_t( X;\xi)\rightsquigarrow Conf (\mathbb R_+)$ %for $Y=\mathbb R, K=(-\infty,0]$ 
is continuous. 
\end{enumerate}
\end{theorem}
\vskip .1in
To understand the relations between this paper and the previous works, \cite {BD11}, \cite {BH} and \cite {B},  the following observations are useful:

1. When $X$ is connected, a topological closed 1-form $\omega$ can be represented  by  a real valued map $f:\tilde X\to \mathbb R,$ called lift of $\omega$  cf. subsection  \ref{SS21},  $\tilde X$ the total space of the  principal covering associated to $\omega.$ This lift is determined up to an additive constant, cf. section \ref{S2} ( i.e. $f_1$ and $f_2$ lifts  of the form $\omega$ implies  $f_1= f_2 +t, t\in \mathbb R$). The configurations $\boldsymbol \delta_r^\omega$ and  $\boldsymbol \gamma_r^\omega$ are derived from the 
integer-valued maps $\delta_r^f$ resp. $ \gamma_r^f$ which have as supports points $(a,b)\in \mathbb R^2$ resp. $\mathbb R^2_+=\{(x, y )\mid  x <y\}$  with $a,b$ critical values of $f.$  
Since the support of $\delta^{f+t}$ resp.  $\gamma^{f+t}$ is the $t-$diagonal translate of the support $\delta^{f}$ resp. $\gamma^{f},$ in order to get  independence on the representative $f,$  one  passes to the quotient spaces $\mathbb R^2/\mathbb R= \mathbb R$ resp. $\mathbb R^2_+/\mathbb R= \mathbb R_+,$ where the quotient  is taken w.r. to the diagonal action action $\mu _{\Delta} (t, (a,b))= (a+t, b+t).$  The support of $\boldsymbol \delta_r^\omega$ and $ \boldsymbol \gamma_r^\omega$ is the image by $p:\mathbb R^2\to \mathbb R, \  p(a,b)= b-a,$  
of the support of $\delta_r^f$ and $\gamma_r^f.$ 

2. In this paper the notation $\gamma^f_r$ refers to the restriction to $\mathbb R^2_+= \{(x,y)\in \mathbb R^2 \mid  y-x>0\}$ of the map denoted by the same letter $\gamma^f_r$ in \cite {B}.  Note that such restriction to $\mathbb R^2_+$ collects information on the so called closed-open bar codes of $f,$ the ones of relevance in the Morse-Novikov theory, while the restriction to $\mathbb R^2_-=\{(x,y)\in \mathbb R^2 \mid y-x<0\}$ collects information on the open-closed bar codes of $f.$ Note also that the open-closed bar codes of $f$ correspond to the closed-open bar codes of $-f$ via the correspondence $(a,b)\to (-b,-a).$ 

3. An interesting example of a tame closed topological 1-form is provided by a simplicial 1-cocycle on a finite simplicial complex. 
 An algorithm to derive the bar codes (i.e. points in the support of $\boldsymbol \delta^\omega$ and $\boldsymbol \gamma^\omega$ with their multiplicity is desirable. This is possible and will be the topic of subsequent work. 
 
 In this paper we write ''$=$" for equality or canonical isomorphism and $\simeq$ for isomorphism, not necessary canonical.
 
 An alternative treatment  via persistent homology of Floer-Novikov theory was proposed by Usher and Zhang cf. \cite {UZ}. Their work has challenged us to extend the results presented in \cite {B} from angle-valued maps to topological closed one forms.

\section{Topological closed one forms and tameness} 
\label{S2}

\subsection {Topological closed 1-form }\label {SS21}

A topological closed 1-form, abbreviated TC1-form, extends the concept of  closed differential 1-form on a smooth manifold  $M$ to an arbitrary topological space $X.$ One way to obtain this is  to view it  as an equivalence class 
of multivalued maps (first definition),  an other way is  to view  it as an  equivalence class of equivariant maps on the associated principal $\mathbb Z^k-$covering (second definition). 
\vskip .2in
{\bf First definition}
\begin{enumerate}
\item A {\bf multi-valued map} is  
a systems $\{U_\alpha, f_\alpha: U_\alpha\to \mathbb R, \alpha\in A\}$ s.t. 
\hskip .1in \begin{enumerate}
\item $U_\alpha$ are open sets with $X=\cup U_\alpha,$   
\item $f_\alpha$ are continuous maps s.t $f_\alpha- f_\beta: U_\alpha\cap U_\beta\to \mathbb R$ is locally constant. 
\end{enumerate}
\item 
Two  multi-valued maps are {\bf equivalent} if put together remain a multi-valued map.
\end{enumerate}
\vskip .1in
\begin{definition}
  A  {\bf TC1-form}   
is an equivalence class of multi-valued maps. 
\end{definition}
A TC1-form $\omega$ determines a cohomology class $\xi= \xi(\omega)\in H^1(X;\mathbb R).$
It suffices  to show that a representative $\{ U_\alpha, f_\alpha: U_\alpha\to \mathbb R\}$ of $\omega$ defines for any continuous path $\gamma: [a,b]\to X$  the number $\int_{\gamma} \omega\in  \mathbb R,$ independent of the homotopy class rel. boundary of $\gamma$ and additive w.r. to juxtaposition of paths.  Indeed,  if $\gamma[a,b]\subset U_\alpha $ for some $\alpha,$ then $\int_{\gamma} \omega= f_\alpha(b)- f_\alpha(a);$ if not, one choses a subdivision of $[a,b],$  $a= t_0 <t_1 <\cdots t_r= b,$ such that $\gamma_i:= \gamma|_{ [t_i, t_{i+1}]}$ lie in some open set $U_\alpha$ and define $\int_{\gamma} \omega:= \sum \int_{\gamma_i} \omega.$ This assignment  defines an homomorphism $\xi(\omega): H_1(X;\mathbb Z)\to \mathbb R,$ equivalently a cohomology class $\xi(\omega).$
\vskip .1in 
One denotes by 
$\mathcal Z^1(X)$ the set of all TC1-forms and by
$\mathcal Z^1(X; \xi):= \{\omega\in \mathcal Z^1(X)\mid  \xi(\omega)= \xi\}.$ Clearly $\mathcal Z^1(X)$ is an $\mathbb R-$ vector space and 
$$\mathcal Z^1(X)=\sqcup_{\xi\in H^1(X;\mathbb R)}  \mathcal Z^1(X;\xi).$$ 
\vskip .1in
In view of this definition, for any $X$ compact ANR  one can find open covers of $X$, $\{U_\alpha, \alpha\in A\}$ with  the properties 
that  $A$ is a finite set  
and  $\overline U_\alpha$ is compact, connected and simply-connected.  
Such cover is called {\it good cover}.
A choice  $x_\alpha\in U_\alpha$ makes  $\omega$ uniquely represented by a multivalued map $\{f_\alpha^\omega: U_\alpha\to \mathbb R\}$ 
with $f_\alpha^\omega(x_\alpha)= 0.$ One  calls  the system $\mathcal U:=\{U_\alpha, x_\alpha, \alpha\in A\}$ with $\{ U_\alpha, \alpha \in A\}$ good cover a {\it base-pointed good cover} of $X.$  

The choice of a base pointed good cover $\mathcal U$ defines for the vector space $\mathcal Z^1(X)$ a complete norm,  
$$||\omega||_\mathcal U:= \sup _{\alpha\in A}  \sup _{x\in \overline U_\alpha} |f^\omega_\alpha (x)|,$$ and then a distance in $\mathcal Z^1(X)$ and implicitly in 
$\mathcal Z^1(X;\xi),$   
$$D(\omega_1, \omega_2)_\mathcal U:= || \omega_1-\omega_2||_{\mathcal U}.$$ 
 Different base-pointed good covers lead to equivalent norms. 
The induced topology on $\mathcal Z^1(X)$ is referred to as the {\it compact open topology}. The subsets $\mathcal Z^1(X;\xi)$ are the connected components of $\mathcal Z^1(X).$  
\vskip .1in
{\it Examples:} 
\begin {enumerate}
\item A  closed differential 1-form, $\omega\in \Omega^1(M)$ $d\omega=0$, defines  a TC1-form.  

\noindent  Indeed, in view of Poincar\'e Lemma, for any $x\in M$ one chooses an open neighborhood  $U_x\ni x$  and $f_x:U_x\to \mathbb R$ a smooth map s.t. $\omega_x|_{U_x}= d f_x.$ The system $\{U_x, f_x:U_x\to \mathbb R\}$ provides a representative of a  TC1-form.% $\omega$.
\item A simplicial 1- cocycle %$\omega$ 
on the simplicial complex $X$ defines  a TC1-form.

If $X$ is a simplicial complex, $\mathcal X_0$ the collection of vertices and $S\subset \mathcal X_0\times \mathcal X_0$ the collection of pairs $(x,y),$ $x,y \in \mathcal X_0$ s.t. $x,y$ are boundaries of a $1-$simplex, then a simplicial 1-cocycle is a map $\delta: S\to \mathbb R$ with the properties $\delta(x,y)= -\delta(y,x)$ and for any three vertices $x,y,z$ with $(x,y), (y,z), (x,z)\in S$ one has $\delta(x,y)+ \delta(y,z) + \delta(z,x)=0.$
The collection of open sets  $U_x,$ $U_x$ the {\it open star} of the vertex $x\in \mathcal X_0,$ and the maps $f_x: U_x\to \mathbb R,$  the linear extensions to  open simplexes of $U_x$ of the map given on the vertices in $U_x$ by $f_x(y):= \delta(x,y)$ and  $f_x(x)=0$  provides a representative of a TC1-form.
% defined by $\delta.$ 
\end{enumerate}
\vskip .1in
{\bf Second definition}.
\vskip .1in
Let 
$\xi\in H^1(X;\mathbb R)= Hom (H_1(X;\mathbb Z), \mathbb R)$ and  $\Gamma = \Gamma (\xi):
= \img  (\xi) \subset \mathbb R.$  If $X$ is a compact ANR then  $\Gamma\simeq \mathbb Z^k$ with $k$ called the {\it degree of irrationality} of $\xi.$

\noindent The surjective homomorphism {$\xi: H_1(X;\mathbb Z)\to \Gamma$} defines the associated $\Gamma-$principal covering, {$\pi: \tilde X_\xi= \tilde X\to X$},  i.e. a free action $\mu:\Gamma\times \tilde X\to \tilde X$ with $\pi $ the quotient map $\tilde X \to \tilde X/\Gamma= X.$ This principal covering is unique up to isomorphism. 
When $X$ is connected so is $\tilde X.$ 

A continuous map 
$f: \tilde X\to \mathbb R$  is $\Gamma-$equivariant if  ${f(\mu(g,x))= f(x)+g}.$ 
\vskip .1in
\begin{definition}
A TC1-form $\omega$ of cohomology class $\xi$ is an equivalency class of continuous 
 $\Gamma-$equivariant  real-valued  maps $f: \tilde X\to \mathbb R$  where        
 $f_1$ is equivalent to $f_2$ iff $f_1 - f_2$ is  locally constant.  
\end{definition}  
 One 
 refers to any representative $f$ in this class as a lift of $\omega.$ Clearly any $\Gamma-$
 equivariant map on a $\Gamma-$ principal covering $\tilde X\to X$ defines a cohomology class in 
 $H^1(X;\mathbb R)= Hom(H_1(X;\mathbb Z),\mathbb R),$ the same for equivalent equivariant maps.  This because  for any continuous path $\gamma:[0,1]\to X$ and $\tilde x\in \tilde X$  with $\pi(\tilde x)=\gamma(0)$ there is a 
 unique $\tilde \gamma:[0,1]\to \tilde X$ with  $\tilde \gamma (0)= \tilde x$ and  $\gamma= \pi\cdot \tilde \gamma$ and, by  taking  $\int _{\gamma}\omega:= f(\tilde \gamma(1))- f(\tilde x)) ,$ one obtains an homomorphism $H_1(X;\mathbb Z)\to \mathbb R.$
Denote by $\mathcal Z^1(X;\xi)$ the set of TC1-forms in the cohomology class $\xi.$ 
In view of this definition the choice of the base point $\tilde x$ in $\tilde X$ (actually one in each connected component if $\tilde X$ is not connected) provides a unique lift  $f^\omega_{\tilde x}:\tilde X\to \mathbb R$ of $\omega$ with $f^\omega_{\tilde x}(\tilde x)=0.$
When $X$ is compact one defines  the distance 
$D(\omega_1, \omega_2)_{\tilde x}$ by 
$$D(\omega_1, \omega_2)_{\tilde x}:= \sup _{y\in \tilde X} | f^{\omega_1}_{\tilde x} (y) - f^{\omega_1}_{\tilde x} (y)|$$ 
which, in view of the compacity of $X$ and of the $\Gamma-$equivariance of the lifts, is a complete metric.  It is not hard to see that different choices of the base point $\tilde x$ lead to equivalent distances  and therefore to the same induced topology with the same collection of Cauchy sequences.
\vskip .1in 
It is not hard to show that  the two definitions of  $\mathcal Z^1(X)$ viewed as vector spaces equipped with complete metrics are equivalent. 
To see this  one chooses  a good cover $\{U_\alpha, \alpha\in A\}$ of $X.$ 

Indeed, a multivalued map $\{f_\alpha : U_\alpha\to \mathbb R\}$   representing $\omega$ (cf. the first definition)  can be modified to an equivalent multivalued map $\{f'_\alpha:U_\alpha\to \mathbb R\}$  by adding appropriate constants on each open set $U_\alpha$ so that $f'_\alpha\cdot \pi_\alpha: \pi^{-1}(U_\alpha)\to \mathbb R$ defines a $\Gamma-$equivariant map on $\tilde X,$ hence a representative of a TC1-form  (cf. the second definition) in the same  cohomology class.

 Conversely, a $\Gamma-$equivariant map  representing $\omega$ 
(in the second definition) and  
a collection of  continuous section  $s_\alpha: U_\alpha\to \pi^{-1}(U_\alpha)$ (i.e. $\pi\cdot s_\alpha= id_{U_\alpha}$) give a  multivalued map $\{f\cdot s_\alpha: U_\alpha\to \mathbb R\}$  representing a TC1-form (first definition) in the same cohomology class.  

It is not hard to check that the identifications above make the distances defined by different choices of base points in $U_\alpha$ and $\tilde X$ equivalent and consequently with the same Cauchy sequences.

\subsection{Weakly tame and tame real-valued maps and topological closed 1-form} \label{SS2}

Fix a field $\kappa$. 
The homology considered is always with coefficients in the fixed field $\kappa,$ (for simplicity in writing omitted from the notation) hence a $\kappa-$vector space. 
For a  continuous map $f:X\to \mathbb R$  denote by: 
\begin{equation*}
\begin{aligned}X^f_{ t}: = f^{-1}((-\infty,t]), &\ X^f_{< t}:= f^{-1}((-\infty,t)),\\ X_f^{t}:= f^{-1}([t,\infty))%= X^{-f}_{-t}
,&\  X_f^{> t}:
 = f^{-1}((t,\infty))%= X^{-f}_{<-t}.
 \end{aligned}
 \end{equation*}

Consider the direct system $\{ H_r(X^f_t): ^fi_t^{t'} (r): H_r(X^f_t)\to H_r(X^f_{t'}), \ t\leq t"\}$ with $^fi_t^{t'}(r)$ the inclusion induced linear maps  and  denote by 
$$H^{inv}_r(X; f):= \underset{t\to -\infty}\varprojlim H_r(X^f_t),  \ \ H^{dir}_r(X; f):=  \underset{t\to \infty}\varprojlim H_r(X^f_t)$$ and  the induced linear maps 
$$\begin{aligned} i^t_{-\infty}=^fi^t_{-\infty}(r): H^{inv}_r(X; f)\to &H_r(X^f_t)\\ i^\infty_t= ^fi^\infty_t(r): H_r(X^f_t) \to &H^{dir}_r(X; f) \\ \mathbb J^f_r:= ^fi^\infty_{-\infty} (r):  H^{inv}_r(X; f)\to  &H^{dir}_r(X; f) .\end{aligned}$$
Often when the decorations $f$ or $r$ are implicit  from the context  they will be dropped off the notation.

If $f_1, f_2: X\to \mathbb R$ are two continuous maps 
with $| f_1- f_2 | <C$  then  the inclusions $X^{f_1} _t \subseteq X^{f_2} _{t+C} \subseteq X^{f_1} _{t +2C} \subseteq X^{f_2} _{t+3C}$  induce 
 \begin{equation*}
 \xymatrix {H_r( X^{f_1} _t)\ar@/^2.0pc/@[][rr]^{^{f_1}i_t^{t+2C}} 
  \ar[r] 
&H_r( X^{f_2} _{t+C})\ar@/_2.0pc/@[][rr]^{^{f_2}i_{t+C}^{t+3C}} 
\ar[r] & H_r( X^{f_1} _{t +2C})\ar[r] &H_r(  X^{f_2} _{t +3C})}
\end{equation*}
which, by passing to homology, implies 
%implies, by passing to homology, 
\begin {equation}\label {E1}
\begin{aligned}H^{inv}_r(X; f_1)=& H^{inv}_r(X; f_2), \\ H^{dir}_r(X; f_1)= &H^{dir}_r(X; f_2), \\ ^{f_1}i^\infty_{-\infty}= ^{f_2}i^\infty_{-\infty} =&
\mathbb J^{f_1}_r= \mathbb J^{f_2}_r.
\end{aligned}
\end{equation}  
\vskip .1in
Suppose that $f^\omega: \tilde X_\xi \to \mathbb R $ is the lift  of TC-1 form $\omega$ of cohomology class $\xi$. The $\Gamma-$equivariance  of $f^\omega$ induces a $\Gamma-$action on the system $\{ H_r(X^f_t), ^fi_t^{t'}\}$ which provides a structure of $\kappa [\Gamma]-$modules on $H_r^{inv} (\tilde X; f^\omega)$ and $H_r^{dir} (\tilde X; f^\omega)$ and make $\mathbb J^{f^\omega}_r$ a $\kappa[\Gamma]-$linear map.  
\vskip .1in
If $f^{\omega_1}, f^{\omega_2}:   \tilde X_\xi \to \mathbb R$ are two lifts of the TC-1 forms $\omega_1, \omega_2 \in \mathcal Z^1(X; \xi),$  then $| f^{\omega_1} (x) - f^{\omega_2} (x)| \leq C= \sup _{x\in  K} | f^{\omega_1}(x) - f^{\omega_2}(x)|,$ where $K $ is  the closure of a fundamental domain of the free action $\mu.$  

Then one can also write
\begin{equation}\label {E2}
\begin{aligned}
H^{inv} (X; \xi ) &:= \  H_r^{inv} (\tilde X; f^{\omega_1})= H_r^{inv} (\tilde X; f^{\omega_2})\\
H^{dir} (X; \xi ) &:= \  H_r^{dir} (\tilde X; f^{\omega_1})= H_r^{dir} (\tilde X; f^{\omega_2})\\
\mathbb J^\xi_r  &:=\  \mathbb J^{f^{\omega_1}}_r= \mathbb J^{f^{\omega_2}}_r
\end{aligned}
\end{equation}
and
 $$\mathbb J^\xi_r : H^{inv}(X;\xi) \to H^{dir} (X;\xi).$$  The map $\mathbb J^\xi_r$ is a $\kappa[\Gamma]-$linear map of $\kappa[\Gamma]-$modules. 
\vskip .1in

{\bf Hypothesis  H :}  $\mathbb J^\xi_r : H^{inv}(X;\xi) \to H^{dir} (X;\xi)$ is injective,
\vskip .1in 
\noindent When $X$  is a compact ANR this hypothesis is satisfied fo all $\xi$ of degree of irrationality $\leq 1$ and satisfied for many pars $(X, \xi)$ with $\xi$ of any degree of irrationality  cf. Corollary \ref {C34}.
  We conjecture that this hypothesis is true for any $\xi$ provided $X$ is a compact ANR.
\vskip .2in
\noindent For $a\in \mathbb R$ let $$R^f_a(r):= \dim H_r(X^f_{a},  X^f_{<a})\in \mathbb Z_{\geq 0} \sqcup \infty,  \quad  \ 
R_f^a(r):= \dim H_r(X^{a}_f, X^{>a}_f)\in \mathbb Z_{\geq 0} \sqcup \infty.$$
The value  $a\in \mathbb R$ is called {\it regular value} (w.r. to $\kappa$ ) if  $R^f_a(r)+ R_f^a(r)= R^f_a(r)+ R^{-f}_{-a}(r)= 0$ for any $r$ and {\it critical value} if not regular. value    
Denote by $CR_r(f) \subset \mathbb R$ the set of critical values of $f$  with the property  $R^f_a(r)+ R_f^a(r)\ne 0$ and  by $CR(f):=\cup _r CR_r(f).$

\begin{definition} \label {D2.6}
A continuous map   $f:X\to \mathbb  R$ is called  tame if :
\begin{enumerate}
\item for any closed interval $I\subseteq \mathbb R$ the subspace $f^{-1}(I)$ is an ANR, in particular $X$ is an ANR.
\item  for any $a\in \mathbb R$ and any $r,$ $R^f_a(r)+ R^a_f(r) <\infty.$ 
%\item  the induced linear map $^fi^\infty _{-\infty}: H^{inv}_r(X; f)\to H^{dir}_r(X; f)$ is injective,  
\item the set $CR(f)$ is countable.
\end{enumerate}
\end{definition}
 Since $\tilde X= \cup \tilde X^f_t$ 
  item 1. implies that 
  $$H^{dir}_r(X; f)= H_r(X),$$
and when $X$ is locally compact (and separable) items 1. and 2. imply item 3.
%\end{enumerate}
\vskip .2in
Let  $\omega$ be  a TC1- form on a connected space $X$ and let $f:\tilde X\to \mathbb R$ be a lift of $\omega.$ The sets $CR_r(f)$ and $CR(f)$ are $\Gamma-$invariant with respect to the action of $\Gamma$ on $\mathbb R$
by translation (recall $\Gamma\subset \mathbb R$) and the action is free. The set of orbits $CR_r(f) / \Gamma$ and $CR(f)/\Gamma$ will be denoted by $\mathcal O_r(f)$ and $\mathcal O(f)$ respectively. 
If $f_1$ and $f_2$ are two lifts of $\omega$ then $f_2= f_1 +t.$ The translation by $t$ provides a canonical bijective map $T_t: \mathcal O(f_1) \to \mathcal O(f_2)$ which preserves the $r-$component $\mathcal O_r(f).$ 
Denote then $\mathcal O(\omega):=\bigcup _{f\in \omega} \mathcal O(f) / \sim$ with $o_1\sim o_2$ ($o_i\in \mathcal O(f_i)$) iff $T_t(o_1)= o_2.$

\begin{definition}\
\begin{enumerate}
\item A  TC1-form $\omega \in \mathcal Z^1(X),$  $X$ a compact ANR, is  weakly tame if one lift  $f^\omega$ and then any other is  tame. %and  the set  $\mathcal O(\omega)$ is finite.
\item A TC1-form $\omega \in \mathcal Z^1(X),$ $X$ compact ANR, is  tame if one lift  $f^\omega $ (and then any other) is tame  and  the set  $\mathcal O(\omega)$ is finite.
\end{enumerate}
\end{definition}

When  $X$ is not connected $\omega$ is weakly tame resp.  tame if its restriction to each component is weakly tame resp. tame.
\vskip .1in
{\bf Examples of  tame TC1-forms}
\begin{enumerate}
\item A locally polynomial \footnote { locally polynomial means that locally there exists coordinates s.t. the coefficients of the form are polynomial functions} closed differential 1-form with all zeros isolated on a closed smooth manifold is  tame.
\item  A generic  simplicial 1- cocycle on a finite simplicial complex 
defines a  TC1-form  which is tame. Here generic means that the 1- cocycle takes nonzero values on all  1- simplexes 
\footnote {The tameness  remains true without the hypothesis {\it all zeros are isolated} in case 1. and {\it generic} in case 2. but via  more elaborated arguments. 
 By similar arguments the restriction of a differential closed 1-form on a manifold $M$  to a compact Thom-Mather stratified subset $X\subset M$  defines a tame  TC1-form on $X$ }.
 \end{enumerate}  
 
  Let us check case 1. for  closed  manifolds. The arguments provided remain true when the manifold is compact and the restriction of $\omega$ to the boundary has no zeros.
 
 Let  $\pi: \tilde M \to M$ be the associated $\Gamma-$principal covering, $ f: \tilde M \to \mathbb R$ a lift  of $\omega,$ $\mathcal X$ the set of zeros of $\omega$ and $\tilde {\mathcal X} = \pi^{-1} (\mathcal X)$ the set of critical points of  $f.$  Note that $\mathcal X$ is finite.  Let $\tilde {\mathcal X}(t):= \tilde {\mathcal X} \cap f^{-1}(t).$  

Observe that $\Gamma$ acts freely on the set $\tilde X$  and the set of orbits of this action is in bijective correspondence to the set  $\mathcal X,$  hence is finite.  Observe also that the restriction of $\pi$ to $\tilde {\mathcal X}(t) $ is injective.
 
If $t\in \mathbb R$  is a regular value then $f^{-1}(t)$ is a codimension one smooth submanifold  and if  $t$ is a critical value then $ f^{-1}(t)$ is a codimension one {\it submanifold with finitely many conic singularities}, as many as the cardinality of $\tilde{\mathcal X}(t).$  
 
 For either case, $t$ regular or critical value,  $f^{-1}(t)$ is a closed subset  which is an ANR and then so is $f^{-1}(I) $ for any closed interval $I.$ This verifies requirement  (1) in Definition \ref{D2.6}.
 
 Note that if $t$ is a regular value $\tilde{M}_t$ is a manifold with boundary with interior $\tilde{M}_{<t }$ hence 
 $H_r( \tilde{M}_t,\tilde{M}_{<t})= 0.$  If $t$ is a critical value then $\tilde{M}_t\setminus \tilde{\mathcal X}(t)$ is a manifold with boundary with interior $\tilde{M}_{<t },$ hence $H_r( \tilde{M}_t,\tilde{M}_{<t})= H_r( \tilde{M}_t, \tilde{M}_t\setminus  \tilde{\mathcal X}(t))= H_r(D_t, D_t\setminus \tilde{\mathcal X}(t))$ with   $D_t= \tilde M_t\cap D,$
 where $D$ is a disjoint union of closed small discs embedded in $\tilde M$, whose interior is a neighborhood of $\tilde{\mathcal X}(t).$  The hypothesis ''local polynomial'' permits to choose such small discs that makes $D_t$ and $S_t= (\partial D)\cap \tilde M_t$  compact ANRs and $D_t\setminus \tilde{\mathcal X}(t)$ retractible by deformation to $S_t.$
 Since $(D_t, S_t)$ is a pair of compact ANRs then $H_r(D_t, D_t\setminus \tilde{\mathcal X}(t))= H_r(D_t, S_t)$ is a vector space of finite dimension.
  This verifies the finite dimensionality of $H_r(\tilde M_t, \tilde M_{<t}).$ The same arguments verify the finite dimensionality 
  of $H_r(\tilde M^t, \tilde M^{>t}),$ hence  the requirement (2) in Definition \ref{D2.6} for $t$ a critical value.  The requirement 3. is obvious in view of the compacity of $M.$

In case 2. the arguments are  similar.  Note that if $t$ is a simplicial regular value for the lift $f$ then $f^{-1}(t)$ has a collar neighborhood inside  the simplicial complex $\tilde X$ in which case $(\tilde X_t, \tilde X_{<t})$ can be treated homologically as  $(\tilde M_t,\tilde M_{<t})$ above. If $t$ is a simplicial critical value, in view of the genericity,  except for a finite set of points $\mathcal V_t= \{x_1, \cdots ,x_k\} \subset f^{-1}(f),$   $f^{-1} (t)\setminus \mathcal V_t$ has a collar neighborhood inside  $\tilde X\setminus {\mathcal V_t}.$ With these observations the homological arguments in the smooth case can be repeated.     
In the above examples  $\mathcal O(\omega)$ is a finite set.

\section {Topology}\label {S3}
\subsection {Configurations of points, collision topology, bottleneck topology}\label {SS31}

Consider a pair $(Y,K)$, $Y$ a locally compact space and $K\subset Y$ a closed subset.

A {\it configuration} of points in $Y$ is a map $\delta: Y\to \mathbb Z_{\geq 0}$ with finite support. The total cardinality of the support is the non negative integer $\sum_{y\in Y} \delta(y).$
Denote by $Conf(Y)$ the set of all configurations of points in $Y$ and by $Conf_N(Y)$ the subset of configurations whose support have total cardinality $N.$

For a configuration $\delta\in Conf (Y\setminus K)$ with support $\supp \ \delta:=\{y_1, y_2, \cdots y_k\}$ and a collection of disjoint open sets $U_1, U_2, \cdots U_k, V$
with $x_i\in U_i, K\subset V$ denote by 
$$\mathcal U(\delta; U_1, \cdots U_k, V):= \{\delta'\in Conf (Y\setminus K) \mid \begin{cases} \supp \ \delta'\subset (\bigcup_{i=1,2, \cdots k} U_i)\cup V\\
\sum_{y\in U_i} \delta'(y)= \delta(y_i) \end{cases} \}$$ 
and for  $\delta\in Conf (Y),$ and $K= \emptyset$ write  $$\mathcal U(\delta, U_1, \cdots U_k):= \mathcal U(\delta, U_1, \cdots U_k), \emptyset).$$  

On the set $Conf_N(Y)$ consider the topology generated by the {\it collections of  neighborhoods}  
\newline $\{\mathcal U(\delta, U_1, \cdots U_k)\}$ of  each $\delta \in Conf_N(Y)$  and refer to it as the {\bf collision topology}. As a topological space $Conf_N(Y)$ identifies to $Y^N/\Sigma_N,$ the quotient of the $N-$fold cartesian product of $Y$ by the group of permutations of $N$ elements.   

 On the set $Conf (Y\setminus K)$ consider the topology generated by the {\it collections of neighborhoods} \newline $\{\mathcal U(\delta, U_1, \cdots U_k,V)\}$ of  each $\delta\in Conf (Y\setminus K)$
and refer to it as the {\bf bottleneck topology}. Note that if $K=\emptyset$ the bottleneck topology and collision topology are  the same.

In this paper we will consider only the case  $Y= \mathbb R$ and $K=(-\infty, 0]$ hence $Y\setminus K= \mathbb R_{+}.$

\subsection {Some algebraic topology of a pair $(X,\omega)$} \label {SS32}

Let $\kappa$ be a field, $X$ a compact ANR,  $\omega$ a tame TC1-form 
in the cohomology class $\xi$  of degree of irrationality $k$ and  $\Gamma= \Gamma(\xi) \subset \mathbb R$ the group defined by $\xi,$ cf. subsection \ref {SS21}. 
Note that if $k\geq 2$ then $\Gamma$ is dense in $\mathbb R$. 

\noindent Let $\tilde X\to X$ be the associated principal $ \Gamma-$covering  with the free action $\mu: \Gamma\times \tilde X\to \tilde X$ and $ f:\tilde X\to \mathbb R$ be a lift  of $\omega.$  
 For any $g\in \Gamma$ the homeomorphism $\mu(g, \cdots): \tilde X\to \tilde X$ induces the isomorphism $\langle g \rangle: H_r(\tilde X)\to H_r(\tilde X)$. 

The map $f$ provides two filtrations of $\tilde X$ indexed by $t\in \mathbb R,$  for $t<t' <t'',$
\begin{equation*}
\begin{aligned}\cdots \tilde X^f_{t} \subset &\tilde X^f_{t'} \subset \tilde X^f_{t''} 
\subset\cdots,\\
\cdots \tilde X_f^{t} \supset &\tilde X_f^{t'} \supset \tilde X_f^{t''}\supset \cdots
\end{aligned}
\end{equation*}
  which induce in homology  the filtrations 
 \begin{equation*}
\begin{aligned}\cdots 
\mathbb I^f_{ t}(r) \subseteq& \mathbb I^f_{ t'}(r) 
\subseteq \mathbb I^f_{ t''}(r)\subseteq \cdots,\\
\cdots \mathbb I_f^{t}(r) \supseteq& \mathbb I_f^{t'}(r)\supseteq \mathbb I_f^{ t"}(r)\supseteq \cdots
\end{aligned}
\end{equation*}
with  
$$\mathbb I^f_t(r):= \img (H_r(\tilde X_t)\to H_r(\tilde X)) \ \rm{and}\    \mathbb I_f^t(r):= \img (H_r(\tilde X^t)\to H_r(\tilde X))$$ 
Clearly $\langle g \rangle (\mathbb I^f_{t}(r)) =\mathbb I^f_{t+g}(r)$ and  
$\langle g\rangle (\mathbb I_f^{t}(r)) =\mathbb I_f^{t+g}(r).$

Note  that:
\begin{enumerate}
\item The $\kappa-$vector space $H_r(\tilde X)$ is actually a  f.g. $\kappa[\Gamma]-$module (since $X$ is a compact ANR) actually a Noetherian module, 
\item $\mathbb I^f_{-\infty}(r):= \cap_{t\in \mathbb R} \mathbb I^f_t(r)$ and $\mathbb I_f^{\infty}(r):= \cap_{t\in \mathbb R} \mathbb I_f^t(r)$ are $\kappa [\Gamma]-$submodules, 
\item $H_r(\tilde X)= \cup _t \ \mathbb I^f_t(r)= \cup _t\  \mathbb I_f^t(r),$
\item $H^N_r(X; \xi):= H_r(\tilde X)/ Tor H_r(\tilde X)$ is a f.g. torsion free  $\kappa[G]-$ module of rank $\beta_{alg, r}^N(X; \xi)$ (i.e. the rank of a maximal free submodule), number referred to as the algebraic Novikov-Betti number or simply as Novikov-Betti number. 

\noindent Note that when $\kappa= \mathbb R$ or $\mathbb C$ then $\beta^N_r(X; \xi)$ equals the $L_2-$Betti number $\beta^{L_2}_r(\tilde X)$ of $\tilde X$ (cf. \cite {Lu}  , Lemma 1.34).
\end{enumerate}

\begin{proposition}\label{P31}\

\begin{enumerate}
\item $Tor H_r(\tilde X)= \mathbb I^f_{-\infty} (r).$
\item $Tor H_r(\tilde X)= \mathbb I_f^{\infty} (r).$
\end{enumerate}
\end{proposition}
\begin{proof}
The $\kappa[\Gamma]-$ module structure of $H_r(\tilde X)$ is given by 
$$(\sum_{i= 1, \cdots k} a_{g_i} g_i) \cdot x:= \sum_{i= 1, \cdots k} a_{g_i}(\langle g_i\rangle  (x))$$ with $a_{g_i}\in \kappa, \ a_{g_i}\ne 0$
and then if $x\in \mathbb I^f_t(r)$ resp. $\mathbb I_f^t (r)$ one has
$$(\sum_{i= 1, \cdots k} a_{g_i} g_i) \cdot x\in \mathbb I^f_{t+\max g_i}(r),\ \rm{resp.}\  (\sum_{i= 1, \cdots k} a_{g_i} g_i) \cdot x\in \mathbb I_f^{t+ \min g_i}(r).$$ 

To check the inclusion $Tor H_r(\tilde X) \subset \mathbb I^f_{-\infty}(r)$ in  item 1. one starts with  $x\in Tor H_r(\tilde X)$ which has to belong by (3.) above to some $\mathbb I^f_t(r).$  Suppose that 
$$(a_{0} g_0 + a_1 g_1 + \cdots a_{k} g_k) \cdot x=0$$ with $g_0 < g_1 < \cdots <g_k,  \ a_{g_i}\ne 0.$ 
 Then $x\in \mathbb I^f_t(r)$ implies  
$$a_k  x= -  (a_{0} (g_0-g_k) + a_1 (g_1-g_k)  + \cdots a_{k-1} (g_{k-1}-g_k)) \cdot x,$$ hence $x\in  \mathbb I^f_{t- (g_k- g_{k-1})}(r).$ By repeating the argument $x\in  \mathbb I^f_{t- n(g_k- g_{k-1})}(r)$ 
 for any $n,$ one derives  $x\in  \mathbb I^f_{-\infty}(r).$
 
Similarly, to check the inclusion $Tor H_r(\tilde X) \subset \mathbb I_f^{\infty}(r),$ one starts with $x\in \mathbb I_f^t(r),$  suppose that $g_0 > g_1 > \cdots >g_k, \ a_{g_i}\ne 0$
and derives $x\in  \mathbb I_f^{t + (-g_k+ g_{k-1})}(r),$ hence  $x\in  \mathbb I_f^{t + n(-g_k+ g_{k-1})}(r)$ 
 for any $n,$ hence  $x\in  \mathbb I_f^{\infty}(r).$
 \vskip .1in
To check that $I^f_{-\infty}(r)\subseteq  Tor H_r(\tilde X)$ and $I_f^{\infty}(r)\subseteq  Tor H_r(\tilde X)$  one uses the fact that $H_r(\tilde X)$ is a f.g. $\kappa[\Gamma]-$module.  If $x\in I^f_{-\infty}(r)$ then there exists an infinite collection of negative $g's$ in $\Gamma,$ say $\cdots <g_r< g_{r-1}< \cdots <g_2<  g_1,$ such that 
$\langle g_r\rangle (x) \in \mathbb I^f_{-\infty}(r)$ and, in view of  the fact that  $\mathbb I^f_{-\infty}(r)$ is f.g.,   a finite collection of elements $P_{r_i}\in \kappa[G], i=1,2, \cdots K$ 
 s.t. $$0= \sum _{1\leq i \leq K}  P_{r_i}\cdot (\langle g_{r_i}\rangle (x))=(\sum _{1\leq i \leq K}  P_{r_i} g_{r_i})\cdot  x.$$ Hence one obtains $x\in Tor H_r(\tilde X).$  
 By a similar argument one concludes that $x\in \mathbb I_f^{\infty}(r)$ implies  $x\in Tor H_r(\tilde X).$ 
\end{proof}

As an immediate consequence one has 
\begin{equation}\label {EEE2}
\begin{aligned}
\mathbb I^f_a(r)\cap\mathbb I_f^{\infty}(r)
= \mathbb I_f^\infty(r)= Tor(H_r(\tilde X)\\
\mathbb I^f_{-\infty}(r)\cap\mathbb I_f^{b}(r)= 
\mathbb I^f_{-\infty}(r)= Tor H_r(\tilde X).
\end{aligned}
\end{equation}
which in view of $\cup_{t\i \mathbb R} \mathbb I^f_t(r) = H_r(\tilde X)$ implies $$\img \ \mathbb J^\xi_r \subseteq Tor H_r(\tilde X).$$ 
\vskip .2 in 

For the purpose of this paper we are interested in deciding  when (conjecturally always) for $\xi \in H^1(X;\mathbb R),$ the linear map $\mathbb J^\xi_r$  is injective . Proposition \ref {P33} below 
provides some partial answer to this question.   To formulate them  one needs some preliminary definitions and observations. 
\begin{itemize}
\item For $p:X \to Y$ a continuous map, \  $X, Y$ compact ANRs  denote by $CR(p)$ the critical values set as 
$$CR(p):= \cap \{  K\subset Y, K\ \rm {compact}\ \mid  p: X\setminus  p^{-1} (K) \to Y\setminus  K\  \rm {is \ a\ fibration}\}$$ 

\item For $f$ a lift of the tame TC-1 form  $\omega \in \mathcal Z (X; \xi),$ $\xi \in H^1(X; \mathbb R),$ $X$ compact ANR, 
(to avoid notational repetitions)  it will be convenient to write also 
$$H_r(\tilde X_{-\infty})\ \rm{for}\   H_r^{inv} (X;\xi)$$   $$H_r(\tilde X_{\infty})\ \rm{for}\  H_r^{dir} (X;\xi),$$ use 
the abbreviations $i_t (r)$ and $\pi_t(r)$  for the canonical maps 
\begin{equation*}
\begin{aligned}
i_t  (r) : = i_t^ \infty (r) : H_r(\tilde X_t) \to H_r^{dir} (X; \xi)\\  
%i^t  (r) : = i^t^ \infty (r) : H_r(\tilde X_t) \to H_r^{dir} (X; \xi)\\
\pi_t(r):=i_{-\infty} ^t(r)  : H^{inv} (X;\xi) \to  H_r(\tilde  X_t)
\end{aligned}
\end{equation*} 
%for the canonical maps  
and the notaions $\iota_t(r)$ and $\iota^t(r)$  for the inclusion induced linear maps 
\begin{equation*}
\begin{aligned}
\iota _t (r) : H_r(\tilde X^f (t) )\to  H_r(\tilde X^f_ t), \\    
\iota ^t (r):  H_r(\tilde X^f (t) )\to  H_r(\tilde X^t_f)   
\end{aligned}
\end{equation*} 
where $X^f(t)= f^{-1}(t).$

Then
for  any $0\ne x\in H_r(\tilde X_t), t\in (-\infty, \infty)$ define
$$\tau_t (x)= \tau(x) := \sup \{s \mid  i^s_t(x)\ne 0\} \subset (t, \infty]$$
and put $$\tau (0)= -\infty.$$
The omission of $t$ in this notation is justified by Item (1) in the following list of properties:

(1): $\tau_{t'}( i_t^{t'} (x))= \tau _t(x)$ provided $t' < \tau_t(x).$
 
(2): $\tau(\lambda x)= \tau(x)$ for $\lambda\in \kappa\setminus 0,$

(3) $\tau (gx)= \tau(x)+g$ for $g\in \Gamma,$

(4) $\tau (x+y) \leq \sup \{ \tau(x) , \tau (y)\}.$ 
 \end{itemize}

As a consequence of (2), (3) (4)  above   one has the following
\begin {obs}\label  {O32} \ 

a) If the elements of the collection $\{x_\alpha\in H_r(\tilde X_t)
\mid t \geq -\infty \}$ satisfy $\tau (x_\alpha)\ne \tau (x_\beta) $ for any $\alpha\ne \beta$ then they are $\kappa-$linearly independents.

b) If $x\in
\mathbb H^{inv} _r(X; \xi)$  with $\tau (x) <\infty$ then the collection of elements $\{ gx \in \mathbb H^{inv}_r(X,\xi) \mid g\in \Gamma\}$  are $\kappa-$linearly independent  and for any $t$ the collection of elements $\{ i^t_{-\infty} (gx) \in H_r(\tilde X_t) \mid  g+\tau(x) >t\}$  are $\kappa-$linearly independent.

c) Moreover for any $g$ with  $g > t - \tau(x)$ and $x$ as in (b) there exists 
$y_g \in \ker (\iota_t(r): H_r( \tilde X(t)) \to H_r (\tilde X))$ s.t. $\iota_t(r) (y_g)= \pi_t (r) (gx)$ and the collection of elements $\{y_g \in H_r (\tilde X(t))\}$  
 are $\kappa-$linearly independent. Consequently, if $\dim (\ker (\iota_t(r): H_r(\tilde X(t))\to H_r(\tilde X_t))$ is finite for at least one $t,$ then $x$ has to be $ 0$. 
\end{obs}  

The proof of Item (c) needs in addition  to (2), (3) and (4)  the Mayer-Vietoris sequence 

 \hskip .5in  $\xymatrix{\cdots \ar[r]&H_r(\tilde X^f(t))\ar[r]^<<<{\iota}& H_r (\tilde X^f_t)\oplus H_r(\tilde X_f^t)\ar[r]^<<< {j} & H_r (\tilde X) \ar[r]&\cdots}$

with $\iota= \iota_t (r)\oplus \iota^t (r),    j= i_t(r) + i^t(r)$  where $i^t(r): H_r(\tilde X^t)\to H_r(\tilde X)$ is the linear map induced by the inclusion $\tilde X_f^t\subset \tilde X.$

\begin {proposition} \label {P33}\ 

Suppose $p: X\to T^k$ is a continuous map, $X$ compact ANR, such that 
$CR(p)$ is contained in a finite collection of disjointly embedded $k-$dimensional disks $D_1, D_2, \cdots D_N   \subset T^k$ 
and $\xi_0\in H^1(T^k;\mathbb R).$  
Then the linear map $\mathbb J_r^\xi$ with $\xi= {p^\ast(\xi_0)}$ is injective. 
\end{proposition}

As a consequence  one has 
\begin {corollary} \label {C34}\

\begin {enumerate} [label = (\alph*)]
\item  For any $\xi\in H^1(X;\mathbb R),$  of degree of irrationality $k\leq 1,$  $X$ compact ANR, the linear map $\mathbb J_r^\xi$ is injective.
\item For any $\xi\in H^1(T^k; \mathbb R)$  the linear map $\mathbb J^\xi_r$ is injective.
\item For any $p: X\to T^k$ with homotopy theoretic fiber a compact ANR  and $\xi_0\in H^1(T^k; \mathbb R)$ 
the linear map $\mathbb J^\xi$ with  $\xi= p^\ast (\xi_0)$ is injective.
\item the construction below provides for any$P$ a compact ANR  (resp. closed manifold)  and any $k$ positive integer a compact ANR $X$  (closed manifold) and $\xi\in H^1(X:\mathbb R)$ with the linear map $\mathbb J^\xi_r$ injective.
\end{enumerate}
\end{corollary} 

{\it Proof of Proposition \ref {P33}:}
Consider the pullback diagram  with $\pi_0$ the canonical $\mathbb Z^k-$ principal covering  
\begin{equation} \label {D1}
\xymatrix{
 \mathbb R^k \ar[r]^{\pi_0} &T^k\\
\tilde X\ar[u]^{\tilde p}\ar[r]^\pi  &X\ar[u]^p}\end{equation}
and 
$\omega_0= \sum_{i=1}^k a_i d\theta_i$ with $\theta_1,= e^{it_1},\cdots , \theta_k= e^{it_k}$ the angular coordinates on $T^k,$   $\{t_1, \cdots t_k\}$ the cartesian coordinates on $R^k,$  and 
$\{a_i\in \mathbb R\}$ a collection of real numbers.  With respect to these coordinates $\pi_0( t_1, \cdots, t_k)= (\theta_1, \cdots, \theta_k).$ Note that any cohomology class $\xi_0\in  H^1(T^k; \mathbb R)$ has a  smooth closed one form  of this type as representative.  
 
The map  
 $f_0(t_1, \cdots t_k)= \sum_{i=1}^k a_i t_i$   
 satisfies $d f_0= \pi_0^\ast (\omega_0)$ and if $CR(p)$ is a finite set then there exists $t$ s.t.
 $\pi_0(f^{-1}(t)) \cap CR(p)= \emptyset.$ Then for $f= f_0\cdot \tilde \pi$  the level $f^{-1} (t)$ is a fibration over a hyperplane $f_0^{-1} (t)$ with compact fibers, hence of finite dimensional homology.  By Observation \ref{O32}  the injectivity of $\mathbb J^\xi_r$ follows.
 
If $CR(p)$ si contained in $\sqcup_{i=1.\cdots, N}  D_i $ consider $u :T^k\to T^k$ a map homotopic to $id_{T^k}$ whose restriction to  $T^k\setminus \sqcup D_i$ is a diffeomorphism on the image and  shrinks each disc to its center.  Clearly for $p' = u\cdot p$ \ $CR(p')$ is a finite set of points,  and since $p^\ast (\xi_0)= {p'}^\ast (\xi_0)=\xi$ the injectivity of $\mathbb J^\xi_r$ holds..  

Note that for any $\xi\in H^1(X;\mathbb R)$ one can produce $p: X\to T^k$ and $\xi_0\in H^1(T^k;\mathbb R)$ s.t. $\xi= p^\ast (\xi_0)$ and if $\xi $ is of degree of irrationality $r$ then one can choose  any $k >r$  with $\xi_0$ of degree of irrationality $r.$ 

q.e.d

 {\it Poof of Corollary (\ref {C34}):}
A simple look to the diagram (\ref{D1}) shows that for  $k=1$ the map $f_0 $ and then $f$ is proper hence the homology of any level $f^{-1}(t)$ is finite  dimensional. 
This verifies (a).

For the Item  (b) , when $X= T^k,$ the levels of $f_0$ are hyperplanes and then  by Observation \ref {O32} the injectivity of  $\mathbb J^\xi_r$ holds .

For Item (c), note first that if $CR(p)=\emptyset $  the statement holds because the levels of $f$ have finite dimensional homology. If $p$ is only {\it fibration up to homotopy},  then  by crossing with a compact contractible space $Y$ and by replacing $p$ by a bundle map $p': X\times Y\to T^k$ homotopic to $p\cdot pr_X,$ one derives the statement from the previous situation. 

For Item (d),  choose $P^d$ a closed $d-$dimensional manifold and let $N^{d+k}$ be a compact $(d+k)-$ dimensional compact manifold with boundary $\partial N^{d+k}=  P^d\times S^{k-1}.$ 
Let $D^k\subset T^k$ be a $k-$ dimensional (closed) disk embedded in $T^k= $ the cartesian product of $k-$copies of $S^1,$  $u": N^{d+k}\to D^k$ be a continuous map which restricted to $\partial N^{d+k}$ is the projection on $S^{k-1}$ and $u' : (T^k\setminus \overset {\circ} {D^k}) \times P^d\to  (T^k\setminus \overset {\circ} {D^k})$ be the projection on  $ (T^k\setminus \overset {\circ}{D^k}).$  Here  $\overset {\circ} {D^k}$ denotes the interior of $D^k.$
Let $M^{d+k}$ be the closed manifold
$$ M^{d+k}:=((T^k\setminus \overset {\circ} {D^k}) \times P^d) \cup_{id_{\partial N}} N^{d+k}$$ obtained  by identifying   the boundaries of the compact  manifolds $(T^k\setminus \overset {\circ}{D^k}\times P^d$ and $N^{d+k}$ and $u:=u'\cup u''.$  The homotopy class of $u$ defines a cohomology class in $H^1(M;\mathbb Z^k)$ 
which together with a chosen  injective homomorphism   $\mathbb Z^k\to \mathbb R,$  provides $\xi \in H^1(M; \mathbb R)$ of degree of irrationality $k.$  By Proposition \ref {P33} the injectivity of $\mathbb J^\xi_r$  holds. 
Indeed the composition $\xymatrix{ {T^k\setminus \overset {\circ} {D^k}}  \ar[r]^{\subset} & M\ar[r]^u &T^k}$ induces an isomorphism in dimension one integral homology hence $u$ induces a surjective homomorphism in dimension one integral homology which makes the hypothesis of Proposition \ref {P33} satisfied. 

\subsection {Novikov-Betti numbers (a topological definition)}\label{SS33}

Recall that for  $f:\tilde X\to \mathbb R,$  a lift of  a tame TC1-form $\omega,$  the vector space $\mathbb I^f_t(r) / \mathbb I^f_{<t}(r)$ is zero when $t$ is a regular value and of finite dimension
when $t$  is a critical value, and the  isomorphism 
$\langle g \rangle :H_r(\tilde X)\to H_r(\tilde X)$  induces an  isomorphism $$\langle g\rangle _t: \mathbb I^f_t (r) / \mathbb I^f_{<t}t(r)\to \mathbb I^f_{t+g}(r) / \mathbb I^f_{(t+g)}(r).$$

Consider the $\kappa-$vector space 
\begin{equation} 
\label {EEEE}
\mathbb N\mathbb H_r(f):= \oplus_{t\in \mathbb R}  \mathbb I^f_t(r) / \mathbb I^f_{<t}(r)
\end{equation}
and note that this sum 

1) involves only components corresponding to $t$ critical values, hence of at most countable many,  

2) is a $\kappa[\Gamma]-$ module  whose multiplication by $g$ is provided by the component-wise isomorphism $\langle g\rangle_t,$    

3) is  independent of the lift $f$ up to an  
isomorphism  since  
$$\mathbb I^f_t(r)/ \mathbb I^f_{<t}(r)  =\mathbb I^{f+c}_{t+c}(r) / \mathbb I^{f+c}_{<(t+c)}(r).$$  

4) is a free $\kappa[\Gamma]-$module  canonically isomorphic to  $V\otimes_{\kappa} \kappa[\Gamma]$ where $V$ is the finite dimensional vector space   $$V= \oplus_{o\in \mathcal O(f)/\Gamma} \mathbb I^f_{a^o}(r) /\mathbb I^f_{<a^o}(r)$$ for any choice $a^o\in o\in \mathcal O_r(f).$  Different choices of $a^o$ lead to isomorphic vector spaces $V.$

One defines 
$$\beta_{top,r}^N(X;\omega):= \dim _\kappa V= \rank  (\mathbb N\mathbb H_r(f))$$  
which will be shown in section (\ref {S6}) 
to be same as  
$$\beta_{alg,r}^N(X; \xi(\omega)):= \sup \{\rank L \mid L \ \rm {free\  submodule\  of} \  H_r(\tilde X) \}.$$

One can provide a similar  definition using  $\mathbb I_f^t $ instead of $\mathbb I^f_t $.  Clearly  $\mathbb I^f_t=  \mathbb I_{-f}^{t}$ and $\mathbb I^f_{<t}= \mathbb I_{-f}^{>t}$
with $-f$ being a lift of the TC1-form $-\omega.$ From the algebraic perspective the first is based on the group $\Gamma$ the second on the group $\Gamma'$ canonically isomorphic o $\Gamma$ by the isomorphism $g\to g'=-g.$  This leads to the same numbers $\beta^N_{alg,r}$ and $\beta^N_{top,r}$.  

Both numbers $\beta_{top,r}^N(X; \omega)$ and $\beta_{alg,r}^N(X; \xi(\omega))$ whose equality is verified in section \ref {S6} are referred to as the Novikov-Betti  numbers.

 \section {The maps 
$ \delta^f_r$ and $ \gamma^f_r$} \label {S4}

In this section  $f:X\to \mathbb R$ will be a tame map, cf. Definition \ref{D2.6}. 
 Recall from the previous section the notations: 

$X^f_{a}:= f^{-1}((-\infty,a]), X^f_{< a}:= f^{-1}((-\infty,a)),  X^f_\infty=  X,$

$X_f^{a}:= f^{-1}([a, \infty)),$  \quad $ X_f^{> a}:= f^{-1}((a, \infty)),$    $X^{-\infty}_f= X,$

$\mathbb I^f_a(r):= \img (H_r(X^f_{a})\to H_r(X)),$ \quad  $\mathbb I^f_{<a}(r)= \img (H_r(X^f_{<a})\to H_r(X)) =\cup_{\alpha<a} \mathbb I^f_\alpha(r),$

$\mathbb I_f^a(r):= \img (H_r(X_f^{a})\to H_r(X)),$ \quad $\mathbb I_f^{>a}(r)= \img (H_r(X_f^{>a})\to H_r(X))= \cup_{\beta>a} \mathbb I_f^\beta(r).$

$\mathbb I^f_{-\infty}(r)= \cap_{a\in \mathbb R} \mathbb I^f_a(r),$ \quad $\mathbb I_f^{\infty}(r)= \cap_{a\in \mathbb R} \mathbb I_f^a(r).$

\noindent For any $a,b \in \mathbb R$  define 
$\boxed{\mathbb F^f_r(a,b):= \mathbb I^f_a(r)\cap \mathbb I_f^b(r)}$ and when $a<b$ define  

\hskip 1in $\boxed{\mathbb T^f_r(a,b):=\ker (H_r( X^f_ a)\to H_r(X^f_ b)).}$ 

\noindent Extend the definitions of  $\mathbb T^f_r(<a,b)$ to:   
\begin{equation}
\begin{aligned} 
\mathbb T_r ( a, \infty)&=\ker (H_r(X^f_a)\to H_r(X)) =  \varinjlim_{a<b\to \infty} \mathbb T_r(a, b),\\ 
\mathbb T_r (<a,b)&=\ker (H_r(X^f_{<a})\to H_r(X^f_b)) =  \varinjlim_{a>a' \to a} \mathbb T_r(a', b)\ \   \rm{for}\  a\leq b\leq \infty\\
\mathbb T_r (a, <b)&= \ker (H_r(X^f_a)\to H_r(X^f_{<b}))=  \varinjlim_
{a<b'\to b} 
\mathbb T_r(a, b')\  \ \rm{for}\ \   a <b' <b \\
\mathbb T_r (-\infty , b)&=\ker(\underset{b>a\to -\infty}\varprojlim H_r(X^f_a) \to H_r(X^f_b))=  \varprojlim_{b>a\to -\infty} 
\mathbb T_r(a,b).
\end{aligned}
\end{equation}
\noindent Note that :    
\begin{enumerate}
\item for $a<b <c\leq \infty$ the obvious linear map 
$$ \ker ( \mathbb T^f_r(a,c)\to \mathbb T^f_r(b,c)) \to  \mathbb T^f_r(a,b)$$ is an isomorphism,   
\item  for $a< b$ the homology exact sequence of the pair $(X^f_{b}, X^f_{a})$ implies the short exact sequence 
\begin{equation} \label {E5}
\xymatrix{0\ar[r]& \coker (H_r( X^f_ a)\to H_r(X^f_b))\ar[r] &\mathbb H_r(X^f_{b}, X^f_{ a}) \ar[r] &\mathbb T^f_{r-1}(a,b)\ar [r]&0,}
\end{equation}
\item  
  the commutative diagram 
{\scriptsize  \begin{equation} \label {E6}
\xymatrix{
         &0\ar[d]&0\ar[d]&\\
0\ar[r]&\mathbb T_r(a,b)\ar[d]\ar[r]^{=} &\mathbb T_r(a,b)\ar[d]\ar[r]&0\ar[d]\ar[r]&0\\
0\ar[r]&\mathbb T_r(a,\infty)\ar[d] \ar[r] &H_r(X_{a})\ar[d]\ar@{->>}[r] &\mathbb I_a(r)\ar[d]\ar[r]&0\\             
0\ar[r]&\mathbb T_r(b,\infty)\ar@{->>}[d]\ar[r] &H_r(X_b)\ar@{->>}[r]\ar@{->>}[d] &\mathbb I_b(r)\ar@{->>}[d]\ar[r]&0\\  
0\ar[r]&\coker (\mathbb T^f_r(a,\infty)\to \mathbb T^f_r(b, \infty))\ar[d] \ar[r] &\coker (H_r( X^f_a)\to H_r(X^f_b))\ar[d]\ar@{->>}[r] &\mathbb I_b(r)/ \mathbb I_a(r)\ar[d]\ar[r]&0\\
&0&0&0  }
\end{equation}}
has all columns and the first three rows exact, which makes the forth row also exact,

\item By passing to direct limit when $a\to b$ the short exact sequence (\ref{E5}) and the last row of (\ref{E6}) imply the isomorphism 
\begin{equation}\label {E7}
H_r(X^f_{ b}, X^f_{< b}) \simeq \mathbb I^f_b(r)/ \mathbb I^f_{ <b} (r) \oplus \coker (\mathbb T^f_r(<b,\infty)\to \mathbb T^f_r(b, \infty))\oplus \mathbb T^f_{r-1}(<b,b).
\end{equation}
\end{enumerate}
\vskip .1in
{\bf The conjecturally trivial vector spaces  $\hat \lambda^f_r$ and numbers $\lambda^\omega_r.$}
 
  For $a\in \mathbb R$   let  $$i_{-\infty}^{<a}(r) : \mathbb T_r(-\infty,a)\to \mathbb T_r(<a,a)$$  be the canonical linear map induced from  $i_t^{<a}(r): \mathbb T_r(t,a)\to \mathbb T_r(<a,a)$ by passing to the inverse limit when $t\to -\infty.$  Since $H_r(\tilde X_a, \tilde X_{<a})$ is  finite dimensional  then so is 
  \begin{equation}\label {Ee} \hat \lambda^f_r(a):=\img (i_{-\infty} ^{<a}(r)).\end{equation} 
 Denote by $$ \lambda^f_r(a):=\dim  \hat \lambda^f_r(a)$$ and 
observe that  
  \begin{enumerate}
 \item  if $h$ is a different lift of $\omega,$ hence $h= f+t,$ then $ \lambda^h_r(a)=  \lambda^f_r(a+t)$
 \item  $ \lambda^f_r(a+g)=  \lambda^f_r(a)$ for any $g\in \Gamma.$
 \end{enumerate}
Then one defines $$ \lambda^f_r(o^f):=  \lambda^f_r(a)$$ for $a\in o^f\in CR(f)/\Gamma,$ and then
\begin{equation} \label {EE11} \lambda^\omega_r:= \sum _{o\in CR(f)/\Gamma} \lambda^f_r(o),\end{equation}  integer number independent of the lift of $\omega.$ 
In case $\omega \in \mathcal Z^1(X;\xi),$\  $\mathbb J_r^\xi$  injective  implies $\lambda^\omega_r=0.$
 
 \subsection  {The assignments $\hat \delta_r^f$ and $\delta^{ f}_r.$ } \label {SS41}
\vskip .1in 

Let  $f: X\to \mathbb R$ be a fixed tame map.  
Since $f$ is fixed the decoration "$f$"  will not appear  the notation below. 

Call {\it box} a subset $B\subset \mathbb R^2$ of the form $B=(a',a]\times [b, b')$ for $-\infty\leq a' < a, \quad  b< b'\leq \infty, $  
and define 
$$\boxed{\mathbb F_r(B):=  
\frac{\mathbb I_a(r) \cap \mathbb I^b(r)} {\mathbb I_{a'}(r)\cap  \mathbb I^b(r)+\mathbb I_a(r)\cap  \mathbb I^{b'}(r)}.
}$$ Let $$ \pi^B_{(a,b)} = \pi^B_{(a,b)}(r):\mathbb F_r(a,b)\to \mathbb F_r(B)$$ be the canonical projection \footnote  {When implicit in the context,
 $r$ will be dropped off the notation}.
\vskip .1in

For $B= B'\sqcup B''$ 
with $B'= (a',a'']\times [b, b')$ ,  \ $B''= (a'',a]\times [b,b'),$\ \   $ -\infty \leq a' < a''< a, \  b< b' < b'' \leq \infty$
or    
with $B'= (a',a]\times [b'', b'),$ \ \ $B''= (a',a]\times [b,b''),$ \ \  $ -\infty \leq a' < a, \  b< b'' <b'\leq \infty$ 
the inclusion $B'\subseteq B$ induces the {\bf injective} 
linear map $$ i_{B'}^B=  i_{B'}^B(r): \mathbb F_r(B')\to \mathbb F_r(B) 
$$ and  the inclusion $B''\subseteq B$ induces the {\bf surjective} linear map $$\pi_{B}^{B''}=\pi_{B}^{B''}(r): \mathbb F_r(B)\to \mathbb F_r(B'').$$

For $-\infty\leq a''< a'< a; \  b< b' < b''\leq \infty$  one denotes by:  
\begin{equation*}
\begin{aligned}
B_{11}:=&(a'',a']\times [b',b''),&  
B_{12}:= (a',a]\times [b',b'')\\
B_{21}:=&(a'',a']\times [b,b'), &
B_{22}:= (a',a]\times [b,b'),\\
\end{aligned}
\end{equation*}
and by
\begin{equation*}
\begin{aligned}
B_{1\cdot}:= &B_{11}\sqcup B_{12}, 
\ B_{\cdot 1}:=&B_{11}\sqcup B_{21}\\
B_{\cdot 2}:= &B_{12}\sqcup B_{22},
\ B_{2\cdot}:= &B_{21}\sqcup B_{22}\\
B :=&B_{1\cdot}\sqcup \ B_{2\cdot}\quad \   B :=&B_{\cdot 1}\sqcup B_{\cdot 2}
\end{aligned}
\end{equation*}

\hskip 1 in \begin{tikzpicture} [scale=1.2]
\draw [<->]  (0,4) -- (0,0) -- (5,0);
\node at (-.2,3.5) {$b''$};
\node at (-0.2,1.5) {$b'$};
\node at (-0.2,1) {$b$};
\node at (1,-0.2) {$a''$};
\node at (3,-0.2) {$a'$};
\node at (4,-0.2) {$a$};
\draw [line width=0.10cm] (1,1) -- (4,1);
\draw [line width=0.10cm] (4,1) -- (4,3.5);
\draw [dashed, ultra thick] (1,1) -- (1,3.5);
\draw [dashed, ultra thick] (1,3.5) -- (4,3.5);
\draw [line width=0.10cm] (1,1.5) -- (4,1.5);
\draw [line width=0.10cm] (3,1) -- (3,3.5);
\node at (2,1.25) {$B_{21}$};
\node at (2,2.5) {$B_{11}$};
\node at (3.5,1.25) {$B_{22}$};
\node at (3.5,2.5) {$B_{12}$};
\node at(2.4, -1) {Figure 1
};
\end{tikzpicture}
In view of the definitions one has 
 \begin{proposition}\label {P41}\

1.  The sequence 
$$\xymatrix{ 0\ar[r]&\mathbb F_r(B')\ar[r]^{i_{B'}^B} &\mathbb F_r(B)\ar[r]^{\pi_B^{B''}}&\mathbb F_r(B'')\ar[r] &0}$$
 is exact.

2. The  diagram 
$$ \xymatrix {      &0\ar[d]             &0\ar[d]       &0\ar[d]\\
                 0 \ar[r]&\mathbb F_r(B_{11})\ar[r]^{i^{B_{1\cdot}}_{B_{11}}}\ar[d]_{i_{B_{11}}^{B_{\cdot 1}}}&\mathbb F_r(B_{1\cdot})\ar[r]^{\pi^{B_{12}}_{B_{1\cdot}}}\ar[d]^{i^B_{B_{1.}} }&\mathbb F_r(B_{12})\ar[r] \ar[d]&0\\
                 0 \ar[r]&\mathbb F_r(B_{\cdot1})\ar[r]^{i^B_{B_{\cdot1}}}\ar[d]&\mathbb F_r(B)\ar[r]^{\pi^{B_{\cdot 2}}_{B}}\ar[d]^{\pi_B^{B_{2\cdot}}} &\mathbb F_r(B_{\cdot 2})\ar[r]\ar[d]^{\pi^{B_{\cdot 2}}_{B_{22}}} &0\\
                 0 \ar[r]&\mathbb F_r(B_{21})\ar[r]^{i^{B_{2\cdot}}_{B_{21}}}\ar[d]&\mathbb F_r(B_{2\cdot})\ar[r]^{\pi^{B_{22}}_{B_{2\cdot}}}\ar[d] &\mathbb F_r(B_{22})\ar[r]\ar[d] &0\\
                                    &0                          &0                             &0}
$$ is commutative with  all rows and columns  exact sequences.
\end{proposition}

One denotes by  $\pi _{B}^{B_{22}}$ for the composition  $\pi _{B}^{B_{22}}:= \pi ^{B_{22}}_{B_{\cdot 2}} \cdot \pi ^{B_{\cdot 2}}_{B} = \pi ^{B_{22}}_{B_{2 \cdot }} \cdot \pi ^{B_{2 \cdot}}_{B}$  and by $i ^{B}_{B_{11}} $  the composition  $i ^{B}_{B_{11}}:= i ^{B}_{B_{1\cdot}} \cdot i ^{B_{1 \cdot}}_{B_{11}} = i ^{B}_{B_{\cdot 1}} \cdot i ^{B_{\cdot 1}}_{B_{11}}$  
and in general by $$\pi_B^{B''}= \pi_B^{B''}(r): \mathbb F_r(B)\to \mathbb F_r(B'')$$ when the box $B''$ is located in the lower-right corner of the box $B$ 
and by $$i_{B'}^{B}= i_{B'}^{B}(r): \mathbb F_r(B')\to \mathbb F_r(B)$$ when the box $B'$ is located in the upper-left corner of the box $B.$ 
The map $\pi_B^{B'}$ is always surjective and $i_{B'}^{B}$ always injective.

 For $\epsilon >0$ one denotes by  $B(a,b;\epsilon)$ the set $B(a,b;\epsilon):= (a-\epsilon]\times [b, b+\epsilon).$  Suppose $-\infty \leq a'<a, \ \ b <b'\leq \infty. $  
 
Define  $$\hat\delta _r(a,b):=\varinjlim_{\epsilon\to 0} \mathbb F_r(B(a,b;\epsilon))$$  
w.r. to the surjective linear maps $\xymatrix{\mathbb F_r(B(a,b;\epsilon))\ar[r]^{\pi^{B(a,b;\epsilon')}_{B(a,b;\epsilon)}}&\mathbb F_r(B(a,b;\epsilon'))},$ $\epsilon> \epsilon',$
 and denote by $$\pi^{(a,b)}_{B}= \pi^{(a,b)}_{B}(r): \mathbb F_r(B)\to \hat \delta_r(a,b)$$   
 the composition $\xymatrix{ \mathbb F_r(B)\ar[r]^{\pi_B^{B(a,b;\epsilon)}} &\mathbb F_r(B(a,b;\epsilon))\ar[r]^-{\pi^{(a,b)}_{B(a,b;\epsilon)}} &\hat \delta_r(a,b)} $ for $\epsilon <\inf \{(a'-a), (b'-b)\}$  
 and by 
$$\pi^{(a,b)}_{\{a,b)\}}=\pi^{(a,b)}_{\{a,b\}}(r): \mathbb F_r(a,b)\to \hat \delta_r(a,b)$$  the composition $\xymatrix{ \mathbb F_r(a,b)\ar[r]^{\pi_{\{a,b\}}^{B(a,b;\epsilon)}} &\mathbb F_r(B(a,b;\epsilon))\ar[r]^-{\pi^{(a,b)}_{B(a,b;\epsilon)}} &\hat \delta_r(a,b).}$  
Both compositions are surjective maps independent on $\epsilon.$

One has 
\begin{equation}\label {E7}
\boxed{\hat\delta_r(a, b)= \frac{\mathbb I_{a}(r)\cap \mathbb I^b(r)}{\mathbb I_{<a}(r)\cap \mathbb I^{b}(r)  +\mathbb I_a(r)\cap \mathbb I^{>b}(r)}}
\end{equation}

Define $$\mathbb F_r((a',a]\times b):=\varinjlim_{\epsilon\to 0} \mathbb F_r((a',a]\times [b, b+\epsilon)),$$
w.r. to the surjective maps 
$\pi_{(a',a]\times [b, b+\epsilon)}^{(a',a]\times [b,b+\epsilon')}(r): \mathbb F_r((a',a]\times [b, b+\epsilon))\to \mathbb F_r((a',a]\times [b, b+\epsilon')), \epsilon >\epsilon'$
and denote by $$\pi_B^{(a',a]\times b}= \pi_B^{(a',a]\times b}(r): \mathbb F_r(B)\to \mathbb F_r((a',a]\times b))$$ 
and 
$$\pi^{(a,b)}_ {(a',a]\times b}= \pi^{(a,b)}_ {(a',a]\times b}(r)=  \mathbb F_r((a',a]\times b))\to \hat\delta_r(a,b)$$ 
 the canonical surjective maps induced by passing to limit when $\epsilon  \to 0$.

One has 
\begin{equation}\label {E8}
\boxed{\mathbb F_r((a',a]\times b)= \frac{\mathbb I_{a}(r)\cap \mathbb I^b(r)}{\mathbb I_{a'}(r)\cap \mathbb I^{b}(r)  +\mathbb I_a(r)\cap \mathbb I^{>b}(r)}.}
\end{equation}
 
Define $$\mathbb F_r(a\times [b,b')):=\varinjlim_{\epsilon\to 0} \mathbb F_r((a-\epsilon,a]\times [b, b'))$$
w.r. to the surjective maps 
$\pi_{(a- \epsilon ,a]\times [b,b')}^{(a-\epsilon' ,a]\times [b,b')}(r): \mathbb F_r((a-\epsilon,a]\times [b, b'))\to \mathbb F_r((a-\epsilon',a]\times [b, b')), \epsilon >\epsilon',$
and denote by  
$$\pi_B^{a\times [b,b')}= \pi_B^{a\times [b,b')}(r): \mathbb F_r(B)\to \mathbb F_r(a\times [b,b'))$$ and 
$$\pi^{a,b}_{a\times [b,b')}= \pi^{a,b}_{a\times [b,b')}(r): \mathbb F_r(a\times [b,b'))\to \hat \delta_r(a,b)$$ the canonical surjective maps 
induced by passing to limit when $\epsilon \to 0.$
 
 One has 
 \begin{equation}\label {E9}
 \boxed{\mathbb F_r(a\times [b,b'))= \frac{\mathbb I_{a}(r)\cap \mathbb I^b(r)}{\mathbb I_{<a}(r)\cap \mathbb I^{b}(r)  +\mathbb I_a(r)\cap \mathbb I^{b'}(r)}.}\end{equation}

 Define 
\begin{equation}\label{E10}
\mathbb F_r(\mathbb R\times b)):=\mathbb I^b(r)/ \mathbb I^{>b}(r)
= \cup_{a\in \mathbb R} \mathbb F_r((-\infty,a]\times b)
\end{equation} 
and

\begin{equation}\label {E11}\mathbb F_r(a\times \mathbb R):=\mathbb I_a(r) /\mathbb I_{<a}(r)
=\cup_{b\in \mathbb R} \mathbb F_r(a\times [b, \infty)) 
\end{equation}
with the inclusions  
$\xymatrix  {\mathbb F_r((-\infty,a']\times b)\ar[r]^{\subseteq} &\mathbb F_r((-\infty,a]\times b)}$ for  $a' < a$
and 
\newline $\xymatrix{\mathbb F_r(a\times [b', \infty))\ar[r]^{\subseteq} &\mathbb F_r(a\times [b, \infty))}$ for  $b' >b$
induced by the linear injective maps $i_B'^B$ cf.  Proposition 2 item 1.

Note  that:
\begin{enumerate}
 \item     $\hat\delta_r (a,b)\ne 0$ implies $a,b\in CR_r(f), \  $
 \item     $\mathbb F_r(a\times \cdots )\ne 0$ implies $a\in CR_r(f),\ $  
 \item     $\mathbb F_r(\cdots \times b )\ne 0$ implies $b\in CR_r(f).$
\item      $\varinjlim_{\epsilon\to 0} \mathbb F_r((a -\epsilon,a]\times b)= \hat\delta_r(a,b),$
\item      $\varinjlim_{\epsilon \to 0} \mathbb F_r(a\times [b,b+\epsilon))= \hat\delta_r(a,b),$
\item      $\varinjlim _{\epsilon, \epsilon'\to 0} \mathbb F_r(a-\epsilon, a]\times [b, b+\epsilon'))= \hat \delta_r(a,b)$.
\end{enumerate}
\vskip .1in
The above definitions  combined with Proposition \ref{P41}  leads to the following proposition.  
\begin{proposition}\label {P43}\ 

\begin{enumerate}
\item
For $-\infty \leq a' <a''< a, \  b\in \mathbb R$ the   
sequence 
\begin{equation}\label {14}
\xymatrix{ 0\ar[r]&\mathbb F_r ((a', a'']\times b))\ar[r]^i &\mathbb F_r((a',a]\times b))\ar[r]^\pi&\mathbb F_r((a'',a]\times b))\ar[r]&0}
\end{equation} 
is exact and for $a\in \mathbb R,  b<b'' <b'\leq \infty $ the  
sequence 
\begin{equation}\label {E15}
\xymatrix{ 0\ar[r]&\mathbb F_r (a\times [b'',b'))\ar[r]^i &\mathbb F_r(a\times [b,b'))\ar[r]^\pi&\mathbb F_r(a\times [b,b'')))\ar[r]&0}
\end{equation} 
is  exact. 

In both sequences  $i$ and $\pi$ are the linear maps induced by the injective linear maps $i_{B'} ^B$ and the surjective linear maps $\pi_B^{B''}.$
\item   
\begin{enumerate}
 \item For any $a\in \mathbb R$ and $ b <b'\leq \infty$ 
  \begin{equation*}
 \begin{aligned}
 \dim \mathbb F_r(a\times  [b,b'))\leq \dim (\mathbb I_a(r)\cap \mathbb I^b(r) /  \mathbb I_{<a}(r)\cap \mathbb I^b(r)) \leq  
  &\dim (\mathbb I_a(r)/  \mathbb I_{<a}(r))\\
   \leq &\dim  H_r(X _{a}, X _{<a})\\
 \end{aligned}
  \end{equation*}  
 and when $a$  is a regular value $\dim \mathbb F_r(a\times  [b,b'))=dim \mathbb F_r(a\times  \mathbb R)= 0.$
 \item For any $b\in \mathbb R$ and $-\infty \leq a' < a $
   \begin{equation*}
 \begin{aligned}
\mathbb F_r((a', a] \times  b)\leq \dim (\mathbb I_a(r)\cap \mathbb I^b(r) /  \mathbb I_a(r)\cap \mathbb I^{>b}(r)) \leq   \dim (\mathbb I^b(r)/  \mathbb I^{>b}(r)) \\
\leq \dim  H_r(X ^b, X ^{>b})\\
 \end{aligned}
  \end{equation*} 
  and when $b$  is a regular value $dim \mathbb F_r((a, a] \times  b)=dim \mathbb F_r(\mathbb R\times b))= 0.$
\end{enumerate}
 The inequalities above hold in view of (\ref{E10} ) and (\ref {E11}) for $[b,b')$ replaced by $[b,\infty)$ or $\mathbb R$ and $(a', a] $ replaced by $(-\infty,a]$ or $\mathbb R$  respectively.

\begin{enumerate} 
\item If either $a$ or $b$ are regular values then $\hat \delta_r(a,b)=0.$ 
\item For any $a\in \mathbb R$  the set $\supp \delta^f_r\cap (a\times \mathbb R)$  is finite  and if $a$ regular value is empty.
\item For any $b\in \mathbb R$  the set $\supp \delta^f_r\cap (\mathbb R\times b)$  is finite  and if $b$ regular value is empty.
\end{enumerate}
\end{enumerate}
\end{proposition}

The relation between the surjective linear maps $\pi_{\cdots}^{\cdots}$ is summarized by   the following commutative diagram with  $B= (a',a]\times [b, b').$

$$\xymatrix{&&\mathbb F_r((a',a]\times b )\ar[rd]^{\pi_{(a'b[\times b}^{(a,b)}}&\\
 \mathbb F_r(a,b) \ar@/^6pc/[rrr]^{\pi^{(a,b)}_{\{a,b\}}} \ar@/^1pc/[rru]^{\pi^{(a',a]\times b}_{\{a,b\}}} \ar[rrd]_{\pi^{a\times [b,b')}_{\{a,b\}}} \ar[r]^-{\pi_{\{a,b\}}^B}
&\mathbb F_r(B)\ar[ru]_{\pi_B^{(a',a]\times b}}\ar[rd]^{\pi_B^{a\times [b,b')}}\ar[rr]^{\pi_B^{(a,b)}}&&\hat\delta_r(a,b)\\
&&\mathbb F_r(a\times [b,b'))\ar[ru]_{\pi_{a\times[b,b')}^{(a,b)}}&.}
$$

\begin{definition} For $(a,b)\in CR_r(f)\times C_r(f)$ a splitting  
$$i_{(a,b)}(r): \hat\delta_r(a,b)\to \mathbb F_r(a,b)\subset H_r(X)$$ 
is a right inverse of the canonical projection  $\pi_{\{a,b\}}^{(a,b)}(r): \mathbb F_r(a,b)\to \hat \delta^f_r(a,b),$\  i.e. $ \pi_{\{a,b\}}^{(a,b)}(r)\cdot i_{(a,b)}(r)= id.$  
 \end{definition}

\medskip 
\vskip .2in 

For $-\infty \leq a' <a ,\ b<b' \leq \infty$ let 
$K$ be either  one of  the following sets:
\begin{enumerate}
\item a bounded or unbounded box $B= (a',a]\times [b, b'),$  
\item a bounded or unbounded  horizontal open-closed interval $I= (a',a]\times b,$  
\item  a bounded or unbounded  closed-open vertical interval $J= a\times [b,b'),$      
\item $a\times \mathbb R,$  
\item $\mathbb R\times b,$ 
\item $\mathbb R^2.$
\end{enumerate}
Call $(a,b)\in \mathbb R^2$   the {\it relevant corner} in case 1.  and  the {\it relevant end} in case 2. or 3. 
The interval $(a',a'']\times b,$ when viewed as a subinterval of $(a', a]\times b,$  is called {\it left (open-closed) subinterval} and  $a\times [b'',b'),$ when viewed as a subinterval of $a\times [b,b'),$ is called {\it upper (closed-open) subinterval}.
\vskip .2in 

For $(a,b)\in K$  a  splitting $i_{(a,b)}(r)$   provides the injective linear map $$i_{(a,b)}^K(r):\hat\delta^f_r(a,b)\to \mathbb F_r(K)$$  defined as follows.  

--  If $(a,b)$ is the relevant corner or the relevant end then $i^K_{(a,b)}(r)$ is  the  composition 
$$i_{(a,b)}^K:= \pi^K_{(a,b)}(r)\cdot i_{(a,b)}(r)$$  with $\pi^K_{(a,b)}(r): \mathbb F_r(a,b)\to \mathbb F_r(K)$  the canonical projection.

--  If $(a,b)$ is not {\it relevant} corner or end then:   

a)  in the case 1., 2., 3. one defines 
$$i_{(a,b)}^K(r): = i_{K'}^K(r) \cdot i_{a,b}^{K'}(r)$$ with  $K'\subset K,$  the only upper-left box,  resp. left subinterval, resp. upper subinterval having $(a,b)$ as the relevant corner resp. end. 
 
 b) in the case 4., 5., 6., one defines $i_{(a,b)}^K(r)$ as the direct limit of $i_{(a,b)}^{K'}(r)$ where $K'$ runs among the subsets of $K$  of the same type and located as upper left box resp. left interval resp. upper interval  which make $i_{K'}^K(r)$ injective. 
\vskip  .1in
Choose a collection of splittings $\mathcal S:=\{i_{(a,b)}(r) \mid (a,b)\in CR(f)\times CR(f),  r\in \mathbb Z_{\geq 0}\}.$ Let  $K$ be a set as in cases in 1., 2., 3., 4., 6. above.  
 and $A\subseteq CR(f)\times CR(f).$ 
 Denote  
$$^{\mathcal S} I_A(r) := \oplus _{\alpha, \beta\in A} \ \   i_{(\alpha, \beta)}(r): \oplus_{(\alpha, \beta)\in A}  \hat\delta_r(\alpha, \beta)\to H_r(X),$$
$$^{\mathcal S} I^K_{A\cap K}(r): = \oplus _{(\alpha, \beta)\in A\cap K} \ \ i^K_{(\alpha, \beta)}(r): \oplus_{(\alpha, \beta)\in A}  \hat\delta_r(\alpha, \beta)\to \mathbb F_r(K).$$

\begin{proposition} \label {P45} \ 

For any choice of $\mathcal S$  the following holds:

\begin{enumerate}
\item The maps $^{\mathcal S}I(r)$ and $^{\mathcal S}I^K_{A\cap K}(r)$ are injective. 
\item If $A= CR(f)\times CR(f)$ and $K$ is in either case 2., 3.. 4., or 5. then 
\newline $^{\mathcal S}I^K_{K}(r):=  ^{\mathcal S}I^K_{ A\cap K}(r)$ is an isomorphism.
\end{enumerate}
\end{proposition}
\begin{proof}\
 
Item 1:
It suffices to verify  the statement for $A$ a finite set. This is done  by induction on cardinality of $A$ as follows. 
When $\sharp A=1$  this follows from the fact that $i_r(\alpha, \beta)$ is a splitting.
When $\sharp (K\cap A) \geq 2 $  one can  write  $K=K_1\sqcup K_2$ with 
 $\sharp (K_i\cap A) <\sharp (K\cap A).$
In view of the definition of $^{\mathcal  S} I^{\cdots} _{\cdots}$the following  diagram is commutative.
$$\xymatrix {0\ar[r]&\mathbb F_r(K_1)\ar[r]^{i^K_{K_1}}&\mathbb F_r(K)\ar[r]^{\pi_K^{K_2}}&\mathbb F_r(K_2)\ar[r] &0\\
0\ar[r]&\oplus_{\alpha. \beta \in A\cap K_1} \hat\delta(\alpha, \beta)\ar[r]\ar[u]_{I_{A\cap K_1}(r)} &\oplus_{\alpha. \beta \in A\cap K} \hat\delta(\alpha, \beta)\ar[r]\ar[u] _{I_{A\cap K}(r)}&\oplus_{\alpha. \beta \in A\cap K_2} \hat\delta(\alpha, \beta)\ar[r]\ar[u]_{I_{A\cap K_2}(r)} &  0
}$$
Since  both raws are short exact sequence with the  left and the right vertical arrows injective  the mid vertical arrow is injective too. 

Item 2: 
In view of the tameness  property of $f,$  if 
 $a\notin CR(f)$ then $\mathbb F_r(a\times \mathbb R)= 0.$ 
If $a\in CR(f)$ then $\mathbb F_r(a\times \mathbb R)= \mathbb I_a(r)/ \mathbb I_{<a}(r)$ in view of (\ref{E11}) and by Proposition \ref{P43} item 2. is of finite dimension.   Denote by $\mathbb S(a,b)= \mathbb I_a  \cap \mathbb I^b/ \mathbb I_{<a}\cap \mathbb I^b$ and observe that 
$$\mathbb F_r(a\times \mathbb R)=\mathbb S(a,-\infty) \supseteq \mathbb S(a, b) \supseteq \mathbb S(a, b')\supseteq \mathbb S(a, \infty)=0$$  for  $b<b'.$
The finite dimensionality of $\mathbb F(a\times \mathbb R)$ implies the existence of a finite collection of critical values $b_1 < b_2 <\cdots b_N$ s.t.
$\mathbb S(a,-\infty)= \mathbb S(a,b_1) \supset \mathbb S(a,b_1)\supset \cdots\supset  \mathbb S(a,b_{N-1})\supset \mathbb S(a,b_N)\supset 0$ and   
in view of the definition of $\mathbb S(a,b)$ one has  $\mathbb S(a,b')= \mathbb S(a,b_i)$ provided $b' \in (b_{i-1}, b_i].$ 
This implies that $\hat \delta( a, b_i)= \mathbb S(a, b_i)/ \mathbb S(a, b_{i+1})$ and therefore  
$$\dim \mathbb F_r(a\times \mathbb R)= \sum _{i=1,2, \cdots N} \dim \hat \delta_r(a, b_i) = \sum _{t\in \mathbb R} \dim \hat \delta_r(a, t).  $$  
 
This   implies that  $^{\mathcal S} I^K_{A\cap K}(r)$is an isomorphism and $$\supp\  \hat  \delta_r\cap a\times \mathbb R= \{(a, b_1), (a,b_2),\cdots (a, b_{N})\}.$$

If $b\notin CR(f)$ then $\mathbb F(\mathbb R\times b)= 0$ 
and if $b\in CR(f)$ then by (\ref {E10}), 
\newline $\mathbb F_r(\mathbb R\times b)= \mathbb I^b (r)/ \mathbb I^{>b}(r)$ and by  Proposition \ref{P43} item 2. of finite dimension.   
Denote by $\mathbb U(a,b)= \mathbb I_a  \cap \mathbb I^b/ \mathbb I_{a}\cap \mathbb I^{>b}$ and observe that 
$$\mathbb F_r( \mathbb R\times b)=\mathbb U(\infty, b) \supseteq \mathbb U(a', b) \supseteq \mathbb U(a, b)\supseteq \mathbb U( -\infty, b)=0.$$ 
for $a' >a .$ The finite dimensionality of $\mathbb F_r(\mathbb R\times b)$ implies the existence of a finite collection of critical values $a_1 >a_2 >\cdots >a_N$  s.t.
\ $\mathbb U(\infty,b)=\mathbb U(a_1, b)\supset \mathbb U(a_2,b)\supset \cdots \supset \mathbb U(a_{N-1},b)\supset \mathbb U(a_N,b)\supset 0$ and in view of the definition of $\mathbb U(a,b)$   
one has $\mathbb U(a',b)= \mathbb U(a_i,b)$ if $a' \in [a_{i+1}, a_i).$ 
Therefore  
$$\dim \mathbb F_r(\mathbb R\times b)= \sum _{i=1,2, \cdots N} \dim \hat \delta_r(a_i, b)=  \sum _{t\in \mathbb R} \dim \hat \delta_r(t, b).$$  This  implies that $^{\mathcal S} I^K_{A\cap K}(r)$  is an isomorphism and 
$$\supp \hat  \delta_r\cap a\times \mathbb R= \{(a_1, b), (a_2, b),\cdots (a_N, b)\}.$$
\end{proof}
\vskip .2in
Define 
$$\boxed{\delta^f_r(a,b):=\dim \hat \delta^f_r(a,b).}$$  
As a  consequence of Proposition \ref {P45}  
for any $a\in \mathbb R$ the set $\supp \ \delta^f_r \cap  (a\times \mathbb R)$ is finite and  of total cardinality
\begin{equation}\label {E14}
\dim \mathbb I^f_a(r)/ \mathbb I^f_{<a}(r)= \sum_{(a,x)\in \supp \delta^f_r \cap (a\times \mathbb R)} \delta^f_r(a,x)
\end{equation} 
 hence  equal to zero  when $a$ is a regular value.
Similarly, for any $b \in \mathbb R$ the set $\supp \ \delta^f_r \cap  (\mathbb R \times b)$ is  finite  of total cardinality
\begin{equation}\label {E15}
\dim \mathbb I_f^b(r)/ \mathbb I_f^{>b}(r)= \sum_{(x,b)\in \supp \delta^f_r \cap (\mathbb R \times b)} \delta^f_r(x,b)
\end{equation} \label {E16}
hence  equal to zero  when $b$ is a regular value.

\subsection{The assignments $\hat \gamma_r^f$ and $\gamma^{ f}_r.$ } \label {SS4}

Call 
{\it box above diagonal}, abbreviated {\it ad-box},   a subset $B\subset \mathbb R^2_{+}= \{(x,y)\mid x<y\}$  of the form $B=(a',a]\times (b',b],$ with  $a'<a \leq b'<b .$ 
As in the previous subsection let  $f: X\to \mathbb R$ be a fixed tame map.  
Since $f$ is fixed the decoration "$f$"  will not appear  the notations below. 

For $a'<a \leq b' <b $ the   inclusions $X^f_{a'}\subset X^f_a \subseteq X^f_{b'}\subset X^f_b$   induce in homology   
the commutative cartesian diagram 
\begin{equation*}
\xymatrix{\mathbb T_r(a',b)\ar[r]^{i_{a'}^a(r)} &\mathbb T_r(a,b)\\ \mathbb T_r(a',b')\ar[r]
 \ar[u]^{\subseteq}\ar[ru]^{u(r)} &\mathbb T_r(a,b')\ar[u]^{\subseteq} .}
\end{equation*}
with the property that 
 $\img \ u(r)= \img \ i_{a'}^a(r) \cap \mathbb T_r(a,b').$   Here  $i_{a'}^a(r): \mathbb T_r(a',b)\to \mathbb T_r(a,b)$ is the restriction of $i_{a'}^a(r):  H_r( X_{a'})\to H_r(X_b).$   In order to avoid heavy notation, when implicit from the context we will simply write $i(r)$ instead of $i_{a'}^a (r).$
\vskip .1in 
Define  
$$\boxed{\mathbb T_r(B):=
\mathbb T_r(a,b)/ (i^a_{a'}(r) (\mathbb T_r(a',b))+ \mathbb T_r(a,b').}$$
Suppose that $B_1, B_2, B$ are three ad-boxes with $B_1\sqcup B_2= B$ in either one of the two  relative positions: 
\newline    $B_1$ the left side ad-box  and $B_2$ the right side ad-box (for example $B= B_{\cdot 2}, B_1= B_{12}, B_2=B_{22}$ in Figure 2),
\newline    $B_1$ the down side ad-box and $B_2$ the upper side ad-box (for example $B= B_{1 \cdot}, B_1= B_{11}, B_2=B_{12}$ in Figure 2).

The inclusion $B_1\subseteq B$  induces the  {\bf injective} linear map $ i_{B_1}^B(r): \mathbb T_r(B_1)\to \mathbb T_r(B)$ and the inclusion $B_2\subset B$ the {\bf surjective} linear map 
$ \pi_{B}^{B_2}(r): \mathbb T_r(B)\to \mathbb T_r(B_2).$ 

One still call {\it ad-box}  the set $(-\infty, a]\times (b', b], \ a\leq b',$ and define 
$$\boxed{\mathbb T_r((-\infty,a]\times (b',b]):=  \underset{a' \to -\infty} \varprojlim \mathbb T_r(a,b)/ (i_{a'}^a(r) (\mathbb T_r(a',b))}$$
For $-\infty\leq a''\leq a'< a; \  b''\leq b' < b$  one considers the ad-boxes  
\begin{equation*}
\begin{aligned}
B_{11}:=&(a'',a']\times (b'',b'], &\ 
B_{12}:=& (a'',a']\times (b',b],\\
B_{21}:=&(a',a]\times (b'',b'],&\ 
B_{22}:=&(a',a]\times ('b,b],\\
\end{aligned}
\end{equation*}
\begin{equation*}
\begin{aligned}
B_{1\cdot}:= B_{11}\sqcup B_{12}, \ &\ %\\
B_{\cdot 1}:=B_{11}\sqcup B_{21} \\
B_{2\cdot}:= B_{12}\sqcup B_{22}, \ &\ %\\
B_{\cdot 2}:=B_{11}\sqcup B_{12}\\
B_{2\cdot}:= B_{21}\sqcup B_{22}, \ &\
B \ \ :=B_{1\cdot}\sqcup B_{2\cdot}= B_{\cdot 1}\sqcup B_{\cdot 2}
\end{aligned}
\end{equation*}

\hskip 1 in 
\begin{tikzpicture} [scale=1.2]
\draw [<-]  (3,4) -- (3,3) -- (3,1);
\draw [->]  (0,1.5) -- (3,1.5) -- (5,1.5);

\node at (-.2,3.5) {$b$};
\node at (-0.2,3) {$b'$};
\node at (-0.2,2.5) {$b''$};
\node at (1,1.3) {$a''$};
\node at (2.3,1.3) {$a'$};
\node at (4,1.3) {$a$};
\draw  [dashed, ultra thick] (1,2.5) -- (4,2.5);
\draw [line width=0.10cm] (4,2.5) -- (4,3.5);
\draw [dashed, ultra thick] (1,2.5) -- (1,3.5);
\draw [line width=0.10cm] (1,3.5) -- (4,3.5);
\draw [line width=0.10cm] (1,3) -- (4,3);
\draw [line width=0.10cm] (2.3,2.5) -- (2.3,3.5);
\node at (2,2.7) {$B_{11}$};
\node at (2,3.2) {$B_{12}$};
\node at (3,2.7) {$B_{21}$};
\node at (3,3.2) {$B_{22}$};
\node at(2.4, 0) 
{Figure 2
};
\end{tikzpicture}

By elementary but tedious arguments, for details see the Appendix below,    
one can show. 
\begin{proposition} \label{P447}\

1. If $B_1, B_2, B$ are ad boxes s.t. $B= B_1\sqcup B_2$ then the sequence 
\begin{equation}\label {E9}
\xymatrix{0\ar[r]&\mathbb T_r(B_1)\ar[r]^{i_{B_1}^B(r)} & \mathbb T_r(B) \ar[r]^{\pi_{B}^{B_2}(r)}& \mathbb T_r (B_2)\ar[r] &0}
\end{equation}
is exact.

2. If $B_{11}, B_{12}, B_{21}, B_{22}$ are ad-boxes as in Figure 2  then the diagram 
$$ 
\xymatrix {               &0\ar[d]             &0\ar[d]       &0\ar[d]\\
                 0 \ar[r]&\mathbb T_r(B_{11})\ar[r]^{i^{B_{1\cdot}}_{B_{11}}}\ar[d]&\mathbb T_r(B_{\cdot 1})\ar[r]^{\pi^{B_{12}}_{B_{1\cdot}}}\ar[d] &\mathbb T_r(B_{21})\ar[r] \ar[d]&0\\
                 0 \ar[r]&\mathbb T_r(B_{1\cdot})\ar[r]^{i^B_{B_{1\cdot}}}\ar[d]&\mathbb T_r(B)\ar[r]^{\pi^{B_{\cdot 2}}_{B}}\ar[d] &\mathbb T_r(B_{2\cdot })\ar[r]\ar[d] &0\\
                 0 \ar[r]&\mathbb T_r(B_{12})\ar[r]^{i^{B_{2\cdot}}_{B_{21}}}\ar[d]&\mathbb T_r(B_{\cdot 2})\ar[r]^{\pi^{B_{22}}_{B_{2\cdot}}}\ar[d] &\mathbb T_r(B_{22})\ar[r]\ar[d] &0\\
                                    &0                          &0                             &0}
$$
is commutative with  all rows and columns exact sequences.
\end{proposition}
As a consequence, 
for any $-\infty \leq a'<a\leq b' <b  \rm \ {and} \ \epsilon>\epsilon'$ the inclusion induced  linear maps: 
\begin{equation*}\pi_{(a',a]\times (b-\epsilon,\  b]}^{(a',a]\times (b-\epsilon', b]}(r):\mathbb T_r((a',a]\times (b-\epsilon\ , b]) \to \mathbb T_r((a',a]\times (b-\epsilon',b]),
\end{equation*} 
\begin{equation*}\pi_{(a-\epsilon ,\ a]\times (b', b]}^{(a-\epsilon',a]\times (b',b]}(r):\mathbb T_r((a-\epsilon,a]\times (b', b]) \to \mathbb T_r((a-\epsilon',a]\times (b',b])\end{equation*}
and 
\begin{equation*}\pi^{B(a,b;\epsilon')}_{B(a,b;\epsilon\ )}(r) :\mathbb T_r(B(a,b;\epsilon))\to \mathbb T_r(B(a,b;\epsilon')),\end{equation*}
where  $B(a,b;\epsilon):= (a-\epsilon]\times (b-\epsilon, b],$  are surjective.

Define 
\begin{enumerate}
\item   for $a < b$  $$\hat \gamma_r^f(a,b): = \varinjlim _{\epsilon \to 0}T_r(B(a,b;\epsilon)),$$
with respect to the maps $\pi^{B(a,b;\epsilon')}_{B(a,b;\epsilon\ )}(r),$
\item  for  
$-\infty \leq 
a'<a <b$ 
$$\mathbb T_r((a',a]\times b)):=\varinjlim_{\epsilon\to 0} \mathbb T_r ((a',a ]\times (b-\epsilon, b]) $$ 
with respect to the maps $\pi_{(a',a]\times (b-\epsilon,\  b]}^{(a',a]\times (b-\epsilon', b]}(r)$ and then observe that  $\varinjlim_{\epsilon \to 0} \mathbb T_r((a-\epsilon, a]\times b)= \hat\gamma_r(a,b).$
\item  for  
$a \leq b' <b $ 
$$\mathbb T_r(a\times (b', b]):=\varinjlim_{\epsilon\to 0} \mathbb T_r ((a-\epsilon, a]\times (b', b])$$ 
with respect to the maps 
$\pi_{(a-\epsilon ,\ a]\times (b', b]}^{(a-\epsilon',a]\times (b',b]}(r)$ and then observe that $\varinjlim_{\epsilon \to 0} \mathbb T_r(a \times (b-\epsilon ,b])= \hat \gamma_r (a,b).$
\end{enumerate}
These maps induce 
for $-\infty\leq a' <a'' < a \leq b'<b''<b < \infty$ the following exact sequences.
\begin{equation}\label{E22}
\xymatrix{ 0\ar[r]&\mathbb T_r ((a', a'']\times b)\ar[r]^{\iota_{I_1}^{I}}&\mathbb T_r ((a',a]\times b)\ar[r]^{\pi_I^{I_1}}&\mathbb T_r ((a'',a]\times b)\ar[r]&0}
\end{equation} 
with $I_1=(a', a'']\times b, I= (a', a'']\times b, I_2=(a'',a]\times b$ and
\begin{equation}\label{E23}
\xymatrix{ 0\ar[r]&\mathbb T_r  (a\times (b',b''])\ar[r]^{\iota_{I_1}^{I}}&\mathbb T_r (a\times (b', b])\ar[r]^{\pi^{i_{I}^{I_2}}}&\mathbb T_r (a\times (b'', b]))\ar[r]&0}
\end{equation} 
with $I_1=(a\times (b',b''], I= a\times (b',b], I_2=a\times (b'',b].$ 

Note that the exactness of (\ref{E22}) for $a'=-\infty$  can be derived by passing to the inverse limit in (\ref {E22}) when  $a' \to -\infty.$ The exactness is maintained  in view of the surjectivity of 
$\mathbb T_r ((a'_1, a]\times b) \to \mathbb T_r( (a_2',a]\times b) $ for $-\infty < a_1' < a_2' < a.$ 

Schematically (\ref {E22}) and (\ref {E23}) can be written as 

\begin{equation}\label {EEE1}
\xymatrix{ 0\ar[r] &\mathbb T_r(I_1)\ar[r] &\mathbb T_r(I) \ar[r]  &\mathbb T_r(I_2)\ar[r] &0} 
\end{equation}
is an exact sequence. 
\vskip .1in
Extend the definitions above to  
 $$\mathbb T_r( (a,b)\times b):=\varinjlim_{b-a >\epsilon \to 0} \mathbb T_r ((a, b-\epsilon]\times b) $$ 
with respect to the maps induced by inclusions $(a,b-\epsilon]\times b$ into $(a, b-\epsilon']\times b$ for $ \epsilon >\epsilon',$
  $$\mathbb T_r(a\times (a, \infty)):=\varinjlim _{a<b\to \infty}\mathbb T_r(a\times (a, b])$$
with respect to the maps induced by inclusions $a\times (a, b_1]$ into $a\times (a, b_2]$ for $a < b_1 < b_2 <b.$ 
 and observe 
 $$\mathbb T_r(a\times (a, b]):=\varprojlim_{b-a>\epsilon \to 0} T_r(a\times (a+\epsilon, b])$$
 with respect to the surjective maps  induced by inclusions $(a\times (a+\epsilon,b])$ into $(a\times (a+\epsilon',b])$ for $ \epsilon >\epsilon'.$

\vskip .1in
The reader can also check that 
$$\mathbb T_r((a, b)\times b) = \varinjlim_{\epsilon \to 0}\mathbb T_r((a, b-\epsilon]\times (b-\epsilon, b])$$
with respect to the linear maps $t_{B_\epsilon}^{B_{\epsilon'}}=  \iota^{B_{\epsilon'}}_{(a, b-\epsilon]\times (b-\epsilon',b]}\cdot \pi_{B_\epsilon}^{(a, b-\epsilon]\times (b-\epsilon',b]}$ where  $B_{\epsilon}= (a', b-\epsilon]\times (b-\epsilon,b]$
and $\epsilon ' <\epsilon.$

One can verify that   
\begin{equation}\label {E17}
\begin{aligned} 
\hat\gamma_r(a, b)=  &\mathbb T_r(a, b) / i(r)(\mathbb T_r(<a, b)) +\mathbb T_r(a,<b)\\
\mathbb T_r(a\times (b',b])= &\mathbb T_r(a, b) /i(r)(\mathbb T_r(<a, b)) +\mathbb T_r(a, b')\\
\mathbb T_r ((a', a]\times b)= &\mathbb T_r(a, b)/ i(r)(\mathbb T_r(a', b)) +\mathbb T_r(a, <b)\\
\mathbb T_r(a\times (a,b])= &\mathbb T_r(a, b) / i(r)(\mathbb T_r(<a, b))\\
\mathbb T_r(a\times (a,\infty))= &\mathbb T_r(a, \infty)) / i(r)(\mathbb T_r(<a, \infty))\\ 
\mathbb T_r((a', b)\times b)= &\mathbb T_r(<b,b) / i(r)(\mathbb T_r(a',b)) \\  
\mathbb T_r((-\infty, b)\times b)=& \mathbb T_r(<b,b) /\hat \lambda^f_r(b)
 \end{aligned}
\end{equation}
and that 
\begin{equation}
\varinjlim_{\epsilon,\epsilon'  \to 0} \mathbb T_r((a-\epsilon] \times (b-\epsilon' ,b])= \hat \gamma_r (a,b).
\end{equation}

The description of the canonical projections $\pi_{\cdots}^{\cdots}$ summarized in the diagram  below is implicit in (\ref{E17}).  

$$\xymatrix{&&\mathbb T_r((a',a]\times b )\ar[rd]^{\pi_{(a'b[\times b}^{(a,b)}}&\\
 \mathbb T_r(a,b) \ar@/^6pc/[rrr]^{\pi^{(a,b)}_{\{a,b\}}} \ar@/^1pc/[rru]^{\pi^{(a',a]\times b}_{\{a,b\}}} \ar[rrd]_{\pi^{a\times [b,b')}_{\{a,b\}}} \ar[r]^-{\pi_{\{a,b\}}^B}
&\mathbb T_r(B)\ar[ru]_{\pi_B^{(a',a]\times b}}\ar[rd]^{\pi_B^{a\times [b,b')}}\ar[rr]^{\pi_B^{\{a,b)}}&&\hat\gamma_r(a,b)\\
&&\mathbb T_r(a\times [b,b'))\ar[ru]_{\pi_{a\times[b,b')}^{(a,b)}}&.}
$$
\vskip .1in
In view of the above definitions  Proposition \ref{P447}  leads to the following 
\begin{proposition}\label {P46} \

\begin{enumerate}
\item
\begin{enumerate}
\item For $a< b' <b$ one has
 $$\dim \mathbb T_r(a \times (b',b])  \leq \dim (\coker  (H_r(X_{<a})\to H_r(X_a)) \leq \dim H_r(X_a, X_{<a})$$ 
\item For $a'<a <b$ one has 
\begin{equation*}
\begin{aligned}
\dim \mathbb T_r((a',a] \times b)  \leq \dim (\coker  (H_{r+1}(X_{<b}, X_a) \to H_{r+1}(X_b, X_a))\\
\leq 
\dim H_{r+1}(X_b, X_{<b})
\end{aligned} 
\end{equation*}
  \end{enumerate}
The same inequalities hold for $(b',b]$ in (a) replaced by $(a,\infty)$ and $(a',a]$ in (b) replaced by $(-\infty, b)$
\item 
\begin{enumerate} 
\item If either $a$ or $b$ are regular values then $\hat \gamma_r(a,b)=0,$ 
\item For any $a,$  $\supp \gamma^f_r\cap (a\times (a,\infty))$  is a finite set and when $a$ regular value is empty.
\item For any $b,$  $\supp \gamma^f_r\cap ((-\infty,b)\times b)$  is a finite set and when $b$ regular value is empty.
\end{enumerate}
\end{enumerate}
\end{proposition}
\begin{proof}\

Item 1.:  To check part  a)
observe that  from the commutative diagram (\ref{D21}) with all raws and columns exact 
\begin{equation}\label {D21}
\xymatrix{
           & 0                                                                         & 0                                           &\\
           &\mathbb T_r(a,b)/ i(\mathbb T_r(<a,b) \ar[r]\ar[u]  &H_r(X_a)/ i (H_r(X_{<a}))\ar[u] & \\
 0 \ar[r]&\mathbb T_r(a,b) \ar[r]\ar[u] & H_r(X_a)\ar[r]\ar[u] & H_r(X_b)\\ 
 0 \ar[r]&\mathbb T_r(<a,b) \ar[r]\ar[u] & H_r(X_{<a})\ar[r]\ar[u]^i& H_r(X_b)\ar[u]^{=}}
 \end{equation} 
one derives the injectivity of $\mathbb T_r(a,b)/ i \mathbb T_r(<a,b))   \to H_r(X_a)/ i H_r(X_<a).$  
In the sequence  
{\small $$\xymatrix {\mathbb T_r(a,b) / (i(r) (\mathbb T_r(<a,b)) + \mathbb T_r(a,b'))& \mathbb T_r(a,b)/ i(r) (\mathbb T_r(<a,b))\ar[l] \ar[ld] \\
H_r(X_a)/ i H_r(X_<a)\ar[r] &H_r(X_a, X_{<a})}$$ }
the  right to left arrow   
is surjective and both the left to right arrow and the left to down arrow   
are injective.
Then, in view of the  finite dimensionality of $H_r(X_a, X_{<a}),$  the statement follows. 

To check  part  b)  one considers 
the commutative diagram (\ref{D22}) with all rows and columns exact 
\begin{equation} \label {D22}
\xymatrix{
0 & 0\\ 
H_{r+1}(X_b, X_a)/ i H_{r+1}(X_{<b}, X_a)\ar[r]\ar[u]&\mathbb T_r(a,b) / i(r) (\mathbb T_r(a,<b))\ar[u] &\\
H_{r+1}(X_b, X_a)\ar[r]\ar[u]&T_r(a,b)\ar[r]\ar[u] &0\\
H_{r+1}(X_{<b}, X_a)\ar[r]\ar[u]^{i}&\mathbb T_r(a,<b)\ar[r] \ar[u]^{i(r)}& 0
}
\end{equation}
and one derives the surjectivity of  $$H_{r+1}(X_b, X_a)/ i H_{r+1}(X_{<b}, X_a)\to \mathbb T_r(a,b) / i(r) (\mathbb T_r(a,<b) .$$
From the long exact sequence of the triple $(X_a \subset X_{<b} \subset X_b)$ one derives the injectivity of  
\newline $H_{r+1}(X_b, X_a)/ i H_{r+1}(X_{<b}, X_a) \to H_{r+1}(X_b, X_{<b}).$
Then, in view of the finite dimensionality 
\newline of $H_{r+1}(X_b, X_{<b}),$
the  diagram
$$\xymatrix {\mathbb T_r(a,b) / (i(r) \mathbb (T_r(a',b)) + \mathbb T_r(a,<b))& \mathbb T_r(a,b)/ i(r)\mathbb (T_r(a,<b))\ar[l]  \\
H_{r+1}(X_b, X_a)/ i H_{r+1}(X_{<b},X_a)\ar[ru]  \ar[r]&H_{r+1}(X_b, X_{<b})}, $$
implies the statement b).  The extension to intervals $(a,\infty), (-\infty, b)$  follows in view of the definitions of $\mathbb T_r(a\times (a,\infty))$ and 
$\mathbb T_r((-\infty, b)\times b).$

\vskip .1in  
Item 2.: follows from Item 1. 
\end{proof}

\vskip .2in 
Suppose that $K$ is one of the  following type of sets. 
\begin{enumerate}
\item a bounded or unbounded ad-box $B= (a',a]\times (b', b],$ $ -\infty\leq a' <a< b, \ \   b'<b\leq \infty,$  
\item a  horizontal open-closed interval $I= (a',a]\times b,$  $ -\infty\leq a' <a ,$
\item  a vertical  open-closed  interval $J= a\times (b',b],$    $ a\leq b'<b,$   
\item $(-\infty, b)\times b,$ 
\item $a\times (a,\infty),$  
\end{enumerate}
and observe  that in view of Proposition \ref{P46},  when  $K$ is of type  2., 3., 4. or 5., the vector space $\mathbb T_r(K)$ has  finite dimension. 

Denote by  $(CR(f)\times CR(f))_+ :=\{ (a,b)\mid a,b\in CR(f), a<b\}.$
\vskip .2in
{\bf Splittings:}
\begin{enumerate} 
\item  For any $(a,b)\in (CR(f)\times CR(f))_+ $   a {\it splitting} $$i_{(a,b)}(r): \hat \gamma_r(a,b)\to \mathbb T_r(a,b)$$  is a right inverse of the canonical projection $\pi^{(a,b)}_{\{a,b\}}(r): \mathbb T_r(a,b)\to \hat \gamma^f_r(a,b).$ 
\item  For any  $K$ in either one of the situations above and $(a,b)\in K$  
one defines $$i_{(a,b)}^K(r): \hat \gamma _r(a,b)\to \mathbb T_r(K)$$  first for the case the point $(a,b)$ is the relevant corner or  vertex  of $K$ then  for an arbitrary point of $K$ as in the previous subsection.
\item  
For a collection of splittings $\mathcal S= \{ i_{(a,b)}(r) \mid (a,b)\in CR(f)\times CR(f)_+ \}, $  set  $A\subset (CR(f)\times CR(f))_+$ and $K$ in one of the cases 1. to  5.  above  one defines  
$$\boxed{^{\mathcal S}I^K_{A\cap K}(r)= \sum _{(\alpha, \beta)\in A \cap K}  i^K_{(\alpha, \beta)}(r): \oplus_{(\alpha, \beta)\in A}  \hat\delta_r(\alpha, \beta)\to \mathbb T_r(K) .}$$
\end{enumerate}
Proposition \ref{P46} can be refined to Proposition \ref{P47}
\begin{proposition}\label {P47}\ 

For any choice of $\mathcal S$
 \begin{enumerate}
 \item The maps $^{\mathcal S}I_A^K(r)$ are injective.
\item If $A= (CR(f)\times CR(f))_+$ and $K$ is either of type  2.(when $-\infty <a'$), 3., 4. or 5.,  then  $^{\mathcal S} I^K_{A\cap K}$ is an isomorphism. 
In particular 

\begin{equation}\label {E25} \mathbb T_r(a\times (b',b])\simeq  \oplus _{b' < t\leq b}\hat\gamma _r(a,t)\  \rm{if}\ a\leq b' <b <\infty
\end{equation}

\begin{equation}\label {E25'} \mathbb T_r(a\times (a, \infty)\simeq  \oplus _{a < t \leq \infty} \hat\gamma _r(a,t)
\end{equation}
and 

\begin{equation}\label {E26} \mathbb T_r((a', a]\times b)\simeq  \oplus _{a' < t\leq a}\hat\gamma _{r}(t,b)\  \rm{if}\ -\infty \leq a' <a <b .\end{equation}

\begin{equation}\label {E26'} \mathbb T_r((b, \infty)\times b)\simeq  \oplus _{b < t_\infty} \hat\gamma _{r}(t,b)\  \rm{if}\ -\infty \leq a' <a <b .\end{equation}
\end{enumerate}
\end{proposition}

\begin{proof}\

Item 1. follows by similar arguments as  in the proof of Proposition \ref{P45} item 1.

Item 2.  We treat $K$ of type 2and 4 as case 1 and type 3 and 5 as case 2. 

For case 1 
 let $b\in \mathbb R$  and 
\begin{equation} 
K= \begin{cases} (a', a] \times b,   -\infty \leq a' <a  <b  \\
(a', b)\times b,       -\infty \leq a' <b.
\end{cases}
\end{equation} 

For case 2
let $a\in \mathbb R$  and 
 \begin{equation} 
K= \begin{cases} a\times (b',b] , a\leq b' <b <\infty  \\
a\times (b',\infty),   a\leq b'.
\end{cases}.
\end{equation} 

We will verify two facts, 

{\bf Fact 1:}
 $\mathbb T_r(K)= 0$ iff $\sup \hat \gamma_r \cap K= \emptyset.$
\vskip .1in 
{\bf Fact  2:}
 $\mathbb T_r(K)\ne 0$ implies in case 1 the existence of a finite 
  collection of real numbers $\{ a_1, a_2, \cdots a_M\}$ s.t.   $a' <a_M  <\cdots <a_2 <a_1 \leq a$ 
  in case 2 the existence of a finite collection of real numbers  $\{ b_1, b_2, \cdots b_N\}, b' < b_1 < b_2 <\cdots b_N\leq b$ and 
   in case 1 $\sup \hat \gamma^f_r \cap K:= \{a_1\times b, a_2\times b, \cdots a_M\times b\} $ in case 2
  $\sup \hat \gamma^f_r \cap K:= \{ a\times b_1, a\times b_2, \cdots, a\times b_N\}.$ 
 
 The verifications use the finite dimensionality of $\mathbb T_r(K)$ and the exact sequence (\ref {E22})  in case 1 and (\ref {E23})  in case 2.
 Since the verifications are essentially the same we will do it in details only  Case 1.  
 
 Observe that by Item 1.  $\mathbb T_r(K)= 0$  implies $\sup \hat \gamma_r \cap K= \emptyset.$ 
and want   to conclude that $\mathbb T_r(K)\ne 0$ implies $\supp \gamma^f_r \cap K\ne \emptyset.$  This will verify Fact 1. 
  
 It suffices to check this for $a'\ne -\infty,$ because if $a'=-\infty,$  by  (\ref{E22}),  there exists $\alpha,$ $-\infty <\alpha <a$ such that $\mathbb T_r (\alpha, a] \times b)\ne 0.$
 If no such  $\alpha$ exists then $\mathbb T_r((-\infty,a]\times b)= 0,$  again by (\ref {E22}), impossible in view of the hypothesis .  

 If $a' \ne -\infty$ one can produce two infinite sequences $\{\alpha'_i\}$ and $\{\alpha_i\}$    $a' \leq \alpha'_1 \leq  \alpha'_2,\leq   \cdots \leq \alpha'_n\leq \alpha'_{n+1} \leq \cdots $ and 
 $a\geq  \alpha_1 \geq  \alpha _2 \geq \cdots \geq \alpha_n \geq \alpha_{n+1}\geq \cdots$ such that 
 
(i)  $\alpha '_n \leq \alpha'_{n+1}   <\alpha _{n+1}\leq \alpha_n,$ 

 (ii) $| \alpha_{n+1} -\alpha_{n+1} | =1/2 |\alpha'_n- \alpha_n |,$
  
  (iii) $\mathbb T_r ((\alpha'_n, \alpha_n] \times b)\ne 0$ for every $n.$
 
Indeed, in view of \ref{E22}) one can take  the interval $(\alpha'_{n+1}, \alpha_{n+1}]$ to be  one of the two intervals $I = (\alpha'_n,  \frac {\alpha_n +\alpha'_n} {2} ]$ or $(\frac {\alpha_n +\alpha'_n} {2} , \alpha_n]$ 
 
 which satisfies $\mathbb T_r(I)\ne 0$  and inductively construct the sequences starting with $\alpha'_1= a'$ and $\alpha_1=a .$ They will  satisfy the properties (i), (ii), (iii).  
In view of (i) and (ii) both sequences are convergent to the real number $\alpha \in (a' , a.]$
Since $\dim \mathbb T_r ((\alpha'_n, \alpha_n] \times b)$ provides a  decreasing sequence of positive integers,  hence constant for $n$ large enough, hence with the vector space  
$\mathbb T_r ((\alpha'_n, \alpha_n] \times b)$ stabilising to a nontrivial vector space $\hat \gamma ^f_r(\alpha, b).$ Hence  $(\alpha, b)\in \supp \gamma^f_r \cap K. $

Case 2 is treated similarly  based on (\ref{E23}) , producing sequences $\{\beta'_n\},$ $\{\beta_n\}$ and ultimately $\beta \in (b', b]$ s.t $(a, \beta)\in \supp \gamma^f\cap K.$
\vskip .1in

To check Fact 2 note that 
Proposition (\ref{P46}) produces a  finite sequence  $\{a_M < \cdots <a_2 <a_1\} $  in case 1 and $ \{b_1 < b_2 <\cdots b_N\}$ in case 2 with  $a_i$  and $b_i$  exactly the points where the map $(a',a]\ni t\rightsquigarrow \dim \mathbb T_r( ((t, a]\times b)$  in case 1  and $(b',b]\ni t \rightsquigarrow \dim \mathbb T_r(a \times (b',  t])$  in case 2 have  discontinuities. 
Clearly both $a_i$ and $b_i$  in view of (\ref {E22}) and ({E23}) are critical values and there are no points  $(\alpha ,b), \alpha \in (a_{i+1}, a_i)$ in case 1 or $(a,\beta), \beta\in (b_i, b_{i+1})$ in case 2 in $\supp \gamma^f_r \cap K.$   
 
Item 1 guaranties the injectivity of $^{\mathcal S}I_{K\cap \sup \hat\gamma_r}^K(r).$ 
A supposed  lack of surjectivity is  ruled out using Fact 1.  The isomorphism $^{\mathcal S}I_{K\cap \sup \hat\gamma_r}^K(r)$ implies (\ref {E26}) and (\ref{E26'}) in case 1 and (\ref{E25}) and ({E25'}) in case 2.

\end{proof}
\vskip .2in

\begin{corollary}\label {C48}\

1. $\coker (H_r(X_{<a})\to H_r(X_a) )\simeq (\oplus _{t\in \mathbb R} \hat \delta_r(a,t)) \bigoplus  (\oplus_{t\in \mathbb R} \hat \gamma _r(a,t))$

2. $\ker (H_r(X_{<b})\to H_r(X_b) )\simeq \mathbb T_r((-\infty, b)\times b)\simeq  \oplus_{t\in (-\infty, a)} \hat\gamma^f_r(t,a)\oplus \hat \lambda^f_r(b)$
\end{corollary}
\begin{proof}\

Item 1.: 
In view of last raw of (\ref{E6} ) one has 
$$\coker (H_r(X_{<a})\to H_r(X_a))\simeq  \mathbb I_a(r)/ \mathbb I_{<a}(r))\oplus Coker  (\mathbb T_r(<a,\infty)\to \mathbb T_r(a,\infty)).$$
 
\noindent In view of the fifth equality (\ref{E17})  and of the equality (\ref{E25}) one has
 $$\coker  (\mathbb T_r(<a,\infty)\to \mathbb T_r(a,\infty))\simeq  \mathbb T_r(a\times (a,\infty))\simeq
 \oplus_{t\in  (a,\infty)} \hat \gamma _r(a,t).$$
In view of  the equality (\ref{E14}) 
$$\mathbb I_a (r)/ \mathbb I_{<a}(r) \simeq \oplus_{t\in \mathbb R} \hat \delta^f_r(a, t).$$
In view of 
the exact sequence last raw of  (\ref{E6})
$$\coker (H_r( X^f_a)\to H_r(X^f_ b))\simeq  \coker (\mathbb T^f_r(a,\infty)\to \mathbb T^f_r(b, \infty)) \oplus \mathbb I^f_b(r)/\mathbb I^f_a(r).$$

\noindent Combining these three isomorphisms  item 1. follows. 
\vskip .1in
Item 2. : In view of the seventh equality (\ref{E17}) one has  

$$\mathbb T_r((-\infty, b)\times b)= \mathbb T(<b,b)/  i(r) \mathbb T_r(-\infty, b) $$
\noindent In view of  Proposition \ref{P47} and formula  (\ref{Ee})
$$\mathbb T_r((-\infty, b)\times b)\simeq \oplus _{t\in (-\infty, b)} \hat \gamma_r(t,b) \oplus  \hat \lambda^f_r(b) $$
Putting together these  isomorphisms one obtains Item 2.

\end{proof}
\vskip .2in
To summarize,   the map $\gamma^f_r:\mathbb R^2\to \mathbb Z_{\geq 0}$  satisfies :

\begin{enumerate}
\item for any $a\in \mathbb R,$ $\supp \ \gamma^f_r \cap  (a \times \mathbb R)$ is a finite  set of total multiplicity   
\begin{equation}\label {E20}
\dim \mathbb T_r(a\times (a,\infty))= \sum_{t\in (a,\infty)} \gamma^f_r(a, t),
\end{equation}
in particular for $a $ regular value is zero.
\item for any $b\in \mathbb R$ the set $\supp \ \gamma^f_r \cap  (\mathbb R \times b)$ is  finite  and of total multiplicity 
\begin{equation}\label{E21}
\dim \mathbb T_r((-\infty,  b)\times b) = \sum_{x\in (-\infty, b)} \gamma^f_r(x,b)  +\lambda^f_r(b), 
\end{equation}
in particular for $b$ regular value  is zero.
\end{enumerate}

\vskip .2in 
As a consequence of Corollary {\ref {C48} one obtains. 
\begin{equation}\label {E27}
\begin{aligned}\dim H_r(X_a, X_{<a}) = \sum_{t\in \mathbb R}  \delta_r(a,t) + \sum_{t\in (a,\infty)}  \gamma _r(a,t)  
+ \sum_{t\in (-\infty, a)}  \gamma_{r-1}(t, a) +\lambda^f_{r-1}(a) 
\end{aligned}
\end{equation} 
\vskip .1in
 
 \subsection{ The case of TC1 form} 
 
Suppose that $f:\tilde X\to \mathbb R$ is the lift of a tame TC1-form  $\omega.$  
Then $f$  $\Gamma-$equivariant map,  the sets $CR_r(f)$ and $CR_r(f)$  are  $\Gamma-$invariant w.r. to the translations by the elements of $\Gamma,$  
 the set $\mathcal O_r(f)= CR_r(f)/\Gamma$ is finite  and the maps $\delta^f_r$ and $\gamma ^f_r$ are $\Gamma-$ invariant  i.e.
\begin{equation*}
\begin{aligned}
\delta^f_r(a,b)= \delta^f_r(a+g, b+g)\\
\gamma_r^f(a,b)= \gamma_r^f(a+g, b+g).
\end{aligned}
\end{equation*} 
In view of (\ref{E14}) and (\ref{E27}), one has  
\begin{corollary} \label {C49}For $f$ a lift of a tame TC1- form $\omega$ and for any choice $a^o\in o\in \mathcal O_r(f)$
\begin{enumerate}
\item \begin{equation*}
\beta^N_{top, r}( X;\omega)= \sum _{o\in \mathcal O(f)} (\sum_{t\in \mathbb R} \delta^f_r(a^o, t))
\end{equation*}
\item \begin{equation*} 
\begin{aligned}
 \sum _{o\in \mathcal O(f)} \dim H_r(\tilde X^f_{a^o}, \tilde X^f_{<a^o})  =  \sum _{o\in \mathcal O(f)}
 (\sum_{t \in \mathbb R}\delta^f_r(a^o, t) + \sum_{t\in \mathbb R_+} \gamma^f _r(a^o, a^o+t)  + \\
 \sum_{t\in \mathbb R_+} \gamma^f_{r-1}(a^o-t, a^o) +  \lambda^f _{r-1}(a^o))
\end{aligned}
\end{equation*}
\end{enumerate}
\end{corollary}

In case $\omega$ is a Morse closed differential 1-form  on a closed manifold $M$  and $f:\tilde M\to \mathbb R$ is a  lift of  $\omega$  then $ \sum _{o\in \mathcal O(f)} \dim H_r(\tilde X_{a^o}, \tilde X_{<a^o})$ is exactly the number of zeros of $\omega$  of Morse index $r.$
Indeed, let $\tilde {\mathcal X}_r$ be the set  of critical points of index $r$ and $\mathcal X_r$ be the finite set of zeros of $\omega$ of Morse 
index $r.$ The group $\Gamma$ acts freely on $\tilde {\mathcal X}_r$ and the orbits of this action identify to $\mathcal X_r.$ 
Let $\pi: \tilde {\mathcal X}_r\to \mathcal X_r$ be the quotient map and let $\tilde {\mathcal X}_{r ,a}:= \tilde {\mathcal X}_{r }\cap f^{-1}(a)$ . Note that the restriction of $\pi$ to $\tilde {\mathcal X}_{r ,a}$ is injective.  

Choose for any $o\in \mathcal O(f)$ a critical value $a^o$ and observe that $\cup_{o\in \mathcal O(f)} \tilde {\mathcal X}_{r, a^o}$ identifies by $\pi$ to $\mathcal X_r.$ 
\newline Classical Morse theory identifies the $\kappa-$vector space generated by $\tilde {\mathcal X}_{r,a}$ with the vector space  $H_r(\tilde M_a, \tilde M_{<a})$
and therefore the cardinality of set $\mathcal X_r$ is the dimension of the vector space 
 $\oplus_{o\in \mathcal O(f)} H_r(\tilde M_{a^o}, \tilde M_{<a^o})$  calculated by Corollary \ref{C49} item 2.,
equivalently the rank of the free $\kappa[\Gamma]-$ module 
$\oplus_{a\in CR(f)} H_r(\tilde M_{a}, \tilde M_{<a}).$

\section{ The configurations $\boldsymbol \delta^\omega_r$ and $\boldsymbol\gamma^\omega_r$}

\subsection{The supports of $\delta^f_r$ and $\gamma^f_r$}

In view of Propositions \ref{P43} and \ref {P46},  the support of $\delta^f_r$ and $\gamma ^f_r$ are located on finitely many diagonal $\Delta^\delta_s$  and $\Delta^\gamma_s,$  for a finite collection of values of $s$  ( $s>0$ in the case of $\gamma^f_r$). Here we denote by $\Delta_s:=\{(x,y)\in \mathbb R^2\mid  y-x=s\}.$   
 One way to produce  these diagonals, 
 goes as follows. Since the procedure is the same for $\delta^f_r$ and $\gamma^f_r$ we describe in details only the case of $\delta^f_r.$
 
Choose one $a^o\in o\in \mathcal O_r(f)$ for each orbit $o.$   Each  point in the set   
$$\supp \delta^f_r\cap a_o\times \mathbb R= \{(a^0, b^1_{a^o})   (a^0, b^2_{a^o})\cdots (a^0, b^{n(o)}) \}$$
defines a diagonal  corresponding to   $s=  b^i_{a^o}- a^o,$  so ultimately one obtains a collection of at most $\sum_{o\in \mathcal O_r(f)} n (o)$ such diagonals.  The index $s$ which appears is always a differences of critical values  of the lift $f.$ In view of the equality 
$\delta^f_r(a,b)= \delta^f_r(a+g, b+g)$ the integer $n(o)$ is independent on the choice of $a^o.$ 
Note  that different orbits $o'$ and $o''$  lead to the same diagonal once the equality $b^i_{a^{o'}}- a^{o'}= b^j_{a^{o''}}- a^{o''}$ holds, so the same diagonal can appear multiple times,
different choices of $a^o$ lead to the same diagonals and different lifts of $\omega$ also lead to the same diagonals  with the same number of occurrences.  Similarly one can produce diagonals 
by choosing  $b_o\in  o$ and  considering  the diagonals corresponding to the points $(b_i^{b_o}, b_o),  1\leq i\leq m(o)$ in the set $\supp \delta_r^f\cap \mathbb R\times b_o$,  hence with $s= 
( b_o- a_i^{b_o}).$   Again the integer $m(o)$ depends only on the orbit $o.$  The outcome by both procedures is expected to be the same as argued below. 

To better explain the number of diagonals and the ''correct multiplicity''  associates with each diagonal the following definitions are of help. 
Note  that both cases $\delta^f_r$ and $\gamma^f_r$ are  similar so one treats fin details only the case of $\delta^f_r$ and one  points out the minor notational differences when the case. 
\vskip .1in 

Two points   $(x,y)\in \mathbb R^2$ and $(x',y')\in \mathbb R^2$ are  $\Gamma-$equivalent,  written 
$(x,y)\sim (x',y')$ iff there exists $g\in \Gamma$ s.t. $x= g+x', y=g+y'.$ 

\begin{definition}\label {D}\
 A subset $\mathcal B^f,$ $\mathcal B^f\subset \supp \delta^f_r\subset CR(f)\times CR(f)$ is called a  {\it BASE} for $\supp\  \delta^f_r$ 
 if the following holds:
\begin{enumerate} 
\item  for any $(\alpha, \beta)\in \supp \delta^f_r $  there exists $g\in \Gamma$  and $(a,b)\in B^f$ such that 
$(\alpha, \beta)=(g+a, g+b),$ 
\item If $(a,b)$ and $(a',b')$ are in $\mathcal B^f$ then $(a,b)\sim (a',b')$ implies $(a,b)= (a',b').$
\end{enumerate}
\end{definition}
In view of 2. the pair  $(a,b)$ and $g$ claimed by 1. are unique.
 
 Observe that the following holds:
 \begin{enumerate}
 \item  
 If $\mathcal B^{f_1}_1$ and $\mathcal B^{f_2}_2$ are two bases for the support of $\delta_r^{f_1}$ and $\delta^{f_2}_r,$ $f_1$ and $f_2$ lifts of $\omega,$ 
 then there exists a canonical bijective correspondence $\theta: \mathcal B_1\to \mathcal B_2$ with the property that $\delta^{f_2}(\theta(a,b)) = \delta^{ f_1}(a,b),$ 
 \item  If for any $o\in \mathcal O_r(f)$ one chooses $ a^o \in o$ and 
 let $b_{a^o}^1, b_{a^o}^2, \cdots  b_{a^o}^{n(o)}$ be all  critical values \footnote {finitely many in view of Proposition \ref {P43}} 
 s.t  $\supp \delta^f_r\cap a^o\times R= \{( a^o, b_{a^o}^1), ( a^o, b_{a^o}^2), \cdots ( a^o,  b_{a^o}^{n(o)})\}$ 
 then the finite collection of points $$\cup _{o\in \mathcal O_r(f)}  \{( a^o, b_{a^o}^1), ( a^o, b_{a^o}^2), \cdots ( a^o,  b_{a^o}^{n(o)})\}$$
provides a base for the support of $\delta^f_r.$  
 Denote this base by $\mathcal B^f(\{ a^o\})$ with $\{a^o\}$ the collection of elements $a^o.$ 
\item  If for any $o\in \mathcal O_r(f)$ one chooses $ b_o\in o$ and 
 let $ a^ {b_o}_1,  a^ {b_o}_2, \cdots  a^ {b_o}_{m(o)}$ be all  critical values 
  s.t $\supp \delta^f_r\cap \mathbb R \times b_o
  = \{(a^ {b_o}_1,b_o),  (a^ {b_o}_2, b_o) \cdots  (a^ {b_o}_{m(o)},b_o)\}$ 
then the collection of points $$\cup _{o\in \mathcal O_r(f)}   \{(a^ {b_o}_1,b_o),  (a^ {b_o}_2, b_o) \cdots  (a^ {b_o}_{m(o)},b_o)\} $$
provides a base for the support of $\delta^f_r.$ 
Denote this base by $\mathcal B^f(\{ b_o\}).$

\item Each element $(a,b)\in \mathcal B^f$ of a base  provides a diagonal $\Delta_s$ with $s=b-a$ and each such diagonal $\Delta _s$ appear  as many time as the number of pairs $\{(a,b)\in \mathcal B \mid b-a=s\}.$ It is convenient to assign to $\Delta_s$ the number 
\begin{equation}
\boxed{\boldsymbol \delta_r ^\omega (s):= \sum_{(a,b)\in B,\ b-a=s}\delta^f_r(a,b)} 
\end{equation}  
which, by item 1. above, is independent of the base $\mathcal B^f.$  
\end{enumerate} 
\vskip .2in

The same definition can be made in case of $\gamma^f$  and provide base $\mathcal B^f$ for $\supp\ \gamma^f_r$}. The  same observations, 1., 2., 3., 4.remain  valid when
$\supp \delta^f_r \cap a\times \mathbb R$ and $\supp\ \delta^f_r \cap \mathbb R\times b$ are replaced by $\supp\ \gamma^f_r \cap a\times (a,\infty)$ and $\supp
\ \gamma^f_r \cap (-\infty, b)\times b$ respectively.

The number  assigned to $\Delta_s$ in the case of $\gamma^f_r$ is 
\begin{equation}
\boxed{\boldsymbol \gamma ^\omega_r (s):= \sum_{(a,b)\in \mathcal B^f, \ b-a=s}\gamma^f_r(a,b)}\\  
\end{equation} 

which, by item 1. is independent on the base $\mathcal B^f$ and the lift $f.$  

Given any lift $f$ of a tame $\omega,$  $\boldsymbol \delta _r^\omega (s)$ and $\boldsymbol \gamma_r ^\omega (s)$ can be calculated  using either a base of the type $\mathcal B(\{a^o\})$
or of type $\mathcal B(\{b_o\})$  and one obtains for any choice of a lift $f$ and any choice  $a^o\in o$ or  $b_o\in o,$  $o\in \mathcal O_r(f)$  the following formulae.

\begin{equation}
 \boldsymbol \delta^\omega_r(s)=
 \sum _{o\in \mathcal O_r(f)} \delta^f_r(a^o, a^o +s)= \sum _{o\in \mathcal O_r(f)} 
\delta^f_r(b_o-s, b_o) 
\end{equation}

\begin{equation*}
\begin{aligned}
\boldsymbol \gamma^\omega_r(s)= 
\sum _{o\in \mathcal O_r(f)} \gamma^f_r(a^0, a^0+s)
=  \sum _{o\in \mathcal O_r(f)}\gamma^f_r(b_o-s, b_o)
\end{aligned}
\end{equation*}

The items 1. and 2. in  Corollary \ref {C49} become
\begin{equation}\label {E36} 
\beta^N_{top, r}(X;\omega) = \sum_{s\in \mathbb R} \boldsymbol \delta^\omega_r(s)
\end{equation}

\begin{equation}\label {E37}
  \sum _{o\in \mathcal O(f)} \dim H_r(\tilde X_{a^o}, \tilde X_{<a^o}) =  \sum_{t\in \mathbb R} \boldsymbol \delta^\omega_r(t) + \sum_{t\in \mathbb R_+} \boldsymbol \gamma ^\omega_r(t) + \sum_{t\in \mathbb R_+}  \boldsymbol \gamma^\omega_{r-1}(t)  + \lambda_{r-1}^\omega
  \end{equation}
\vskip .2in 

\subsection {Derivation of the main results} \label{S6}
Proof of Theorem (\ref{TT}):

Item 1. follows from  equality (\ref{E36}) combined with the equality of $\beta^N_{alg,r}(X;\xi(\omega))= \beta^N_{top,r}(X; \omega)$ established by  Theorem \ref{T36} below. 

Item 2. follows from equality (\ref{E37}).

\begin {theorem}\label {T36}
if $\omega$ is a tame TC1-form  then $\beta^N_{top,r}(X;\omega)= \beta^N_{alg,r}(X;\xi(\omega)).$
\end{theorem}

Observe that in view of Proposition \ref{P31} for $x\in H_r(\tilde X)$  either there exists $a= \alpha(x) \in \mathbb R$ s.t. $x\in \mathbb I_a(r)\setminus \mathbb I_{<a}(r)$ or $x\in \cap_{t\in \mathbb R} \mathbb I_t(r)= TorH_r(\tilde X).$ 
Clearly $x\in \mathbb I_a(r)\setminus \mathbb I_{<a}(r)$ implies $\hat x,$ the image of $x$ in  $\mathbb I_a(r)/\mathbb I_{<a}(r),$ is not zero.

Observe that: 

\begin{enumerate}\label {F}
\item  $\alpha(x)\in CR(f)_r,$
\item $\alpha(g\cdot x)\equiv  \alpha(\langle g\rangle (x)) = g+\alpha(x),$
\item $\alpha(x+y)= max\{\alpha(x), \alpha(y)\}$ if $\alpha(x)\ne \alpha(y)$ or if $\alpha(x)=\alpha(y)$ and $\hat x+\hat y\ne 0,$ 
\item $\alpha(x+y) <\alpha (x)= \alpha(y)$ if $\alpha(x)= \alpha(y)$ and $\hat x+\hat y=0 $
\end{enumerate}

Suppose that $e_1, e_2, \cdots e_N$  is a base of  a free  $\kappa[\Gamma]-$module $\mathcal E$ submodule of  $H_r(\tilde X).$ 
Note that the multiplication with elements in $\Gamma$ of any of $e_i's$ does not change their status of remaining together a base for $\mathcal E$, but modifies $\alpha(e_i)$ as indicated in (2.) above.  
 In view of the above properties of $\alpha(x)$ one can modify this base  into a base of $\mathcal E$ consisting of  
$$
\begin{aligned}
 e_{1,1}, e_{1,2}, \cdots e_{1,n_1} \\
 e_{2,1}, e_{2,2}, \cdots e_{2,,n_2} \\
 \cdots \\
e_{r,1}, e_{r,2}, \cdots e_{r,n_r}
\end{aligned}
$$
with the  following properties: 
\begin{enumerate}
\item $N= n_1 + n_2 +\cdots n_r,$  
\item  $\alpha (e_{i,j})= a_i \in CR_r(f),$  
\item  $a_1 > a_1> \cdots > a_r$ with $a_i \in o_i$ different orbits of $CR_r(f)/\Gamma= \mathcal O_r(f),$ i.e. $\Gamma-$independent.
\end{enumerate}
First one observes that for any $i=1,2, \cdots r$ the set $\hat e_{i,1},  \hat e_{i,2}, \cdots \hat e_{i,n_i} $ are $\kappa-$linearly independent elements in $\mathbb I_{a_i}/ \mathbb I_{<a_i}.$ 

Indeed if for a fixed $i$  one has  $\sum_j \lambda_j \hat e_{i,j}=0\  \lambda_i\in \kappa,$ then 
$\alpha (\sum_j \lambda_j e_{i,j})<a_i$ by (4), and  then $\sum_j  \lambda_j  e_{i,j} =\sum_j  Q_j  e_{i,j} + \sum_ {j, s\ne i} P_{s,j} e_{s,j}$ where  $Q_j\in \kappa[\Gamma]$ 
contains only negative elements in $\Gamma $  (i.e. in $\Gamma \cap (-\infty,0)). $ Then one obtains $\sum _j  \lambda_j (1- Q_j) e_{i,j} - \sum_ {j, s\ne i}P_{s,j} e_{s,j}=0,$   
hence $\lambda_j (1- Q_j)=0,$ and because $(1- Q_j)\ne 0$ $\lambda_j=0.$ 

Second, In view of $\Gamma-$independence of $a_i$ the entire collection $\{\hat e_{i,j}\}$ consists of elements $\kappa [\Gamma]-$linearly independent in the $\kappa[\Gamma]-$module
$\oplus_{a\in \mathbb R} \mathbb I_a/ \mathbb I_{<a},$ cf subsection 3.3. This implies $N\leq \beta^N_{top,r}(X;\omega)$  hence  $\beta^N_{alg,r} (X;\xi(\omega)) \leq \beta^N_{top, r}(X;\omega).$

The inequality $\beta^N_{top, r}(X;\omega) \leq \beta^N_{alg,r} (X;\xi(\omega)$ follows from the injectivity of $^{\mathcal S} I(r) $ established in Proposition \ref{P45}. 
provided that a $\Gamma-$compatible collection of splittings is chosen and  such collection exists. 
A collection of splittings 
$i_{a,b}(r):\hat \delta^f_r(a,b)\to \mathbb F_r(a,b)\subseteq H_r(\tilde X)$ is $\Gamma-$ compatible if for any $g\in \Gamma$ one has 
\begin{equation}\label {EE} \langle g\rangle \cdot i_{a,b}(r) = i_{a+g, b+g}(r)\cdot \langle g\rangle_{a,b} \end{equation} 
where  the isomorphism $\langle g\rangle_{a,b} :\hat \delta^f_r(a,b)\to \hat \delta^f_r(a+g, b+g)$ is induced by the isomorphism $\langle g\rangle.$  The choice of an arbitrary collection of splittings $i_{a,b}(r)$ for $(a,b)\in B^f$ a base for $\supp\  \delta^f_r,$  which obviously exists, defines via formula (\ref{EE}) a family of compatible $\Gamma-$splittings.  If the collection $\mathcal S$ is $\Gamma-$ compatible then the $\kappa-$linear map $ ^{\mathcal S}I(r)$ is actually an injective  $\kappa[\Gamma]-$ linear map  from the free $\kappa[\Gamma]$ module of rank $\dim \sum_{a,b\in B^f} \delta^f_r(a,b) = \beta^N_{top,r}(X;\omega) $ to $H_r(\tilde X)$, hence  $\beta^N_{top, r}(X;\omega) \leq \beta^N_{alg,r}(X, \xi(\omega)).$
\vskip .2in

Theorems \ref{TP} and \ref {TS} in the generality stated will be proven in part II and part III of this work.
However in case $\omega$ is of degree of irrationality $0,$ hence $\Gamma=0$ and then $\tilde X=X,$ they follow from Theorems 5.2 and 5.3  
in \cite {B} in view of the fact that $\boldsymbol\delta^\omega(t)= \sum_a \delta^f_r(a, a+t).$ 
In case $\omega$ is of degree of irrationality $1$ and the TC1-form determined by an angle valued map $ f: X\to \mathbb S^1= \mathbb R/2\pi \mathbb Z$  then they follow  from 
\cite {B}   
Theorems  5.5, 5.6, 6.3 and 6.4  by observing that $\boldsymbol \delta ^\omega_r(t)=\sum _{z\in \mathbb C\setminus 0\mid \ln  |z|= t} \delta^f_r(z)$ and $\boldsymbol \gamma ^\omega_r(t)= 
\sum _{z\in \mathbb C\setminus 0 | \ln  |z|= t, |z|>1}\delta^f_r(z).$ If $\omega$ is of degree if irrationality one then  $\Gamma\simeq Z$ with the 
positive generator, a  real number $l\in \mathbb R_+.$ One repeats the  arguments in \cite{B} with $2\pi$ replaced by $l$ and one derives the two 
results from the same theorems in \cite{B}. Reference \cite{B} actually reproduces results in \cite {BH}.

\section {Appendix.1}

{\bf Proof of Proposition \ref{P447}}

Consider commutative  diagrams of $\kappa-$vector spaces $$
\mathbb D= \xymatrix { E_2\ar[r]^{i_2} &F_2\\
E_1\ar[u]^{j_E}\ar[ru]^u\ar [r]^{i_1}&F_1\ar[u]^{j_F}}$$
which satisfies the following three properties
\begin{enumerate}

\item $j_E$ and $j_F$ are injective
\item $j_E: \ker i_1 \to \ker i_2$ is an isomorphism
\item $\img(u) = \img i_2 \cap \img j_F= \img u$
\end{enumerate}
and define $T(\mathcal B):= F_2/ i_2(E_2) + j_F(F_1)$ 

Consider  the diagram  

 \begin{equation}\label {E211}
\xymatrix{A_2\ar[r]^{i^A_2}\ar@/^2pc/[rr]_{i_2}&B_2\ar[r]^{i^B_2}&C_2\\
A_1\ar[u]_{j_A}\ar@/_2pc/[rr]^{i_1}\ar[ur] \ar[r]_{i^A_1} \
\ar[urr]& B_1\ar[u]_{j_B}\ar[ur]\ar[r]_{i^B_1}&C_1\ar[u]_{j_C}}
\end{equation}
and  observe that 

$\bf O1:$

If  each of the three  diagrams $\mathbb B_1, \mathbb B_2, \mathbb B,$ associated with  (\ref{E211}),     

$\mathbb B_1$ with vertices $A_1, A_2, B_1, B_2$,

$\mathbb B_2$ with vertices $B_1, B_2, C_1, C_2$, and 

$\mathbb B$ with vertices $A_1, A_2, C_1, C_2$

\noindent satisfy the properties (1) (2) (3) above 
then (\ref{E211})  induces  the  exact sequence  

\begin{equation*}
\xymatrix@C-5pc { B_2/ (i^A_2(A_2)+ j_B(B_1))= \mathbb T(\mathbb B_1)\ar[rd]^{i}&&  \\
0\ar[u]& \mathbb T(\mathbb B)= C_2/ (i_2(A_2)+ j_C(C_1)) \ar[rd]^{p}&0\\ 
&&C_2/ (i^B_2 (B_2)+ j_C(C_1))= \mathbb T(\mathbb B_2)\ar[u] }
\end{equation*}
with  $i$ induced  by $i^B_2,$  
(well defined because $\img ( i^B_2\cdot j_B)\subseteq \img j_C$)
 $p$  induced by the inclusion $(i_2(A_2)+ j_C(C_1)) \subseteq (i^B_2 (B_2)+ j_C(C_1)).$

\vskip .1in 
Clearly  $p$ is surjective and $p\cdot i=0.$  Property (3) 
implies $i$ injective.  Properties (1), (2) (3) imply that  the sequence is exact.

Similarly consider the diagram 
\begin{equation}\label {E212}
\xymatrix {A_3\ar[r]^{i_3}& B_3\\ 
A_2\ar[r]^{i_2}\ar[ru]\ar[u]^{j_2^A}&B_2\ar[u]_{j_2^B}\\
A_1\ar@/^2pc/[uu]^{j^A}\ar[u]^{j_1^A}\ar[ur]\ar[ruu]\ar[r]^{i_1}&B_1\ar@/_2pc/[uu]_{j^B}\ar[u]_{j_1^B}
}
\end{equation}
and observe by the same arguments that 

$\bf O2:$   

If each of the three  diagrams $\mathbb B_1, \mathbb B_2, \mathbb B,$  associated with  (\ref{E212}),    

 $\mathbb B_1$ with vertices $A_2, A_3, B_2, B_3$,

$\mathbb B_2$ with vertices $A_1, A_2, B_1, B_2$, and 

$\mathbb B$ with vertices $A_1, A_3, B_1, B_2$

\noindent satisfy the properties (1) (2) (3) of the diagram $\mathbb D$ 
Then (\ref{E212})  induces  the  exact sequence  
\begin{equation*}
\xymatrix @C-5pc{B_2/ (i_2(A_2)+ j^B_1(B_1))=\mathbb T(\mathbb B_1)\ar[rd]^{j}&&  \\
0\ar[u]& \mathbb T(\mathbb B)= B_3/ (i_3(A_3)+ j^B(B_1)) \ar[rd]^{p}&0\\ 
&&B_3/ (i_3 (A_3)+ j_2^B(B_2))= \mathbb T(\mathbb B_2)\ar[u].}
\end{equation*}

Note that any ad-box $B= (a'a]\times(b',b]$ $ < a' <a <b\leq b' <b $  defines a diagram $D$ as above with $E_2= \mathbb T_r(a',b), F_2= \mathbb T_r(a,b), 
E_1= \mathbb T_r(a',b'), F_1= \mathbb T_r(a,b')$ and $i_1,i_2, j_E, j_F$ the induced linear maps.
 The ad-boxes $B_{12}, B_{22}, B{\cdot 2}$, $B_{11}, B_{21}, B_{\cdot 1}$ and $B_{1,\cdot}, B_{2,\cdot}, B$  are in the situation provided by $\bf O 1$, and the boxes  
$B_{11}, B_{12}, B{1,\cdot }$, $B_{21}, B_{22}, B_{2,\cdot }$ and $B_{\cdot 1}, B_{\cdot 2}, B$ in the situation provided by $\bf O 2$.
Consequently Proposition 5 follows.

\end{document}

\vskip .1in

\section {Appendix 2}

\begin{proposition}
If $f:X\to \mathbb R$ is a weakly tame mapthen the obvious linear maps \begin{equation} \label {EE1}
I_a(r): H_r(X^f_a) \to \underset {\epsilon \to 0} \varprojlim H_r(X_{a+\epsilon}) \end{equation}
 and 

\begin{equation}\label {EE2}
I_a(r): H_r(X^f_a, X^f_{<a}) \to \underset {\epsilon \to 0 }\varprojlim H_r(X_{a+\epsilon},X^f_{a+\epsilon})  
\end{equation}

are injective.

\end{proposition}

\begin{proof}  Consider $D_i\subset X, i=1,2, \cdots $  compact ANRs   with nonempty interiors s.t 
\begin{enumerate}
\item  $\bigcup Int D_i= X $
\item $D_i \subset  Int  D_{i+1}, $
\end{enumerate}

hence  for $a\leq b$ 
\begin{equation} 
\begin{aligned}
H_r(X)= &\underset {i\to\infty} \varinjlim H_r(X\cap D_i)\\
H_r(X_a)= &\underset {i\to\infty} \varinjlim H_r(X_a\cap D_i), \\
H_r(X_{<a})= &\underset {i\to\infty} \varinjlim H_r(X_{<a}\cap D_i), \\
H_r(X_a, X_b)=& \underset {i\to\infty} \varinjlim H_r(X_a\cap D_i, X_b\cap D_i),\\      %,\ H_r(X_b)= \underset {i\to\infty} \varinjlim H-r(X\cap D_i)\\
H_r(X_a, X_{<b})= &\underset {i\to\infty} \varinjlim H_r(X_a\cap D_i, X_{<b}\cap D_i) %H_r(X)= \underset {i\to\infty} \varinjlim H-r(X\cap D_i)
\end{aligned}
\end{equation}

and for $i<j$ consider the following diagram 
{\scriptsize
\begin{equation}
\xymatrix{
H_r(\tilde X_a)\ar[rr]^{I_a}&&\underset{\epsilon\to 0}\varprojlim H_r(\tilde X_{a+\epsilon}) \ar[rr] ^{\pi_{\epsilon}}&& H_r(\tilde X_{a+\epsilon})&\\
&H_r(\tilde X_a\cap D_j)\ar[ul]^{i_{D_i}}\ar[rr]^{I_{a,j}}&&\underset{\epsilon\to 0}\varprojlim H_r(\tilde X_{a+\epsilon}\cap D_j) \ar[ul]^{'i_{D_j}}\ar[rr] ^{'\pi_{\epsilon}}&& H_r(\tilde X_{a+\epsilon}\cap D_j)\ar[ul]\\
H_r(\tilde X_a\cap D_i)\ar[uu]^{i_{D_i}}\ar[ur]^{i_{D_i}^{D_j}}\ar[rr]^{I_{a,i}}&&\underset{\epsilon\to 0}\varprojlim H_r(\tilde X_{a+\epsilon}\cap D_i)\ar[uu] \ar[ur]\ar[rr] ^{"\pi_{\epsilon}}&& H_r(\tilde X_{a+\epsilon}\cap D_i)\ar[uu]^{''i_{D_i}}\ar[ur]^{"i_{D_i}^{D_j}}&
}
\end{equation}}

Here are a few observation about this diagram under the specified hypotheses:
\begin{enumerate}
\item  For any $i$  $I_{a,i}$ is an isomorphism; this because for any $i$ there exists $\epsilon (i)$ s.t. %for $\epsilon (i)$ 
the inclusion induced linear map $H_r(\tilde X_a)\to H_r(\tilde X_{a+\epsilon})$ is an isomorphism  ($\tilde X_{a+\epsilon} \cap D_i$ is a compact ANR.
\item For any $x\in H_r(\tilde X_a)$ there exists  a pair $(i, x_i\in H_r(\tilde X_a \cap D_i)$ s.t. $i_{D_i} (x_i)= x.$
\item For any $y_i\in H_r(\tilde X_a\cap D_i)$ which satisfies  $i_{D_i}(y_i)=0$ there exists $j>i$ s.t $i_{D_i}^{D_j}(y_i)=0.$
\end{enumerate}
Combining these on concludes that if $x\in H_r(|tilde X_a)$ satisfy $I_a(x)=0$ then $x=0.$ 
\vskip .1in

Indeed for such $x$ there exists by item 2 an $x_i$ with $''i_{D_i} \cdot {''\pi_\epsilon} \cdot I_{a,i})(x_i)=0$  and by item 3  there exists $j>i$ s.t $''i_{D_i}^{D_j} \cdot {''\pi_\epsilon} I_{a,i}(x_i)=0$  and any 
$\epsilon.$ 
Then by item  1 for any $\epsilon <\epsilon ( j)$    one has $i_{D_i}^{D_j} (x_i)=0$ hence $x=0.$  This establish (\ref{EE1})

For (\ref {EE2}) the arguments are exactly the same by replacing $\tilde X_a$ resp, $\tilde X_{a+\epsilon}) $ by  the pairs $(\tilde X_a, \tilde X_{<a})$ resp. $(\tilde X_{a+\epsilon}, \tilde X_{<a}).$ 
\end{proof}
\end{document}